\documentclass[a4paper]{amsart}

\usepackage{amsmath,amssymb,amsthm,amsfonts}
\usepackage{mathtools}
\usepackage[UKenglish]{babel}
\usepackage{microtype}
\usepackage[pdftex]{color}
\usepackage[bookmarks=true,hyperindex,pdftex,colorlinks, citecolor=MidnightBlue,linkcolor=MidnightBlue, urlcolor=MidnightBlue]{hyperref}
\usepackage[dvipsnames]{xcolor}

\usepackage[margin=2.5cm]{geometry}
\usepackage{bbm}

\newcommand{\R}{\mathbb{R}}
\newcommand{\T}{\mathbb{T}}
\newcommand{\M}{\mathfrak{M}}

\DeclareMathOperator{\diam}{diam\,}
\DeclareMathOperator{\co}{co}
\DeclareMathOperator{\e}{e}

\renewcommand{\geq}{\geqslant}
\renewcommand{\leq}{\leqslant}

\newcommand{\restricted}{\mathord{\upharpoonright}}
\newcommand{\norm}[1]{\left\Vert#1\right\Vert}
\newcommand{\Free}{{\mathcal F}}
\newcommand{\lip}{{\mathrm{lip}}_0}

\newcommand{\Lip}{{\mathrm{Lip}}_0}

\newcommand{\SA}{\operatorname{SNA}}
\newcommand{\NA}{\operatorname{NA}}
\newcommand{\F}[1]{\mathcal{F}(#1)}

\newcommand{\cco}{\overline{\operatorname{co}}}
\newcommand{\ext}[1]{\operatorname{ext}\left(#1\right)}
\newcommand{\strexp}[1]{\operatorname{str-exp}\left(#1\right)}
\newcommand{\preext}[1]{\operatorname{pre-ext}\left(#1\right)}
\newcommand{\dent}[1]{\operatorname{dent}\left(#1\right)}
\newcommand{\Mol}[1]{\operatorname{Mol}\left(#1\right)}

\newtheorem{theorem}{Theorem}[section]
\newtheorem{proposition}[theorem]{Proposition}
\newtheorem{corollary}[theorem]{Corollary}
\newtheorem{lemma}[theorem]{Lemma}
\newtheorem{fact}[theorem]{Fact}
\theoremstyle{definition}
\newtheorem{definition}[theorem]{Definition}
\newtheorem{example}[theorem]{Example}
\newtheorem{remark}[theorem]{Remark}

\title[Examples and applications of the density of strongly norm attaining Lipschitz maps]{Examples and applications of the density of strongly norm attaining Lipschitz maps}

\dedicatory{Dedicated to Jean Bourgain, \emph{in memoriam}}

\date{August 27th, 2019. Revised January 16th, 2020}

\author[Chiclana]{Rafael Chiclana}
\author[Garc\'ia-Lirola]{Luis C.\ Garc\'ia-Lirola}
\address[Garc\'ia-Lirola]{Department of Mathematical Sciences, Kent State University, Kent OH 44242, USA}
\email{lgarcial@kent.edu}
\urladdr[(Garc\'{\i}a-Lirola)]{\url{http://research.kent.edu/~lgarcial/}}

\author[Mart\'in]{Miguel Mart\'in}
\author[Rueda Zoca]{Abraham Rueda Zoca}
\address[Chiclana, Mart\'in, Rueda Zoca]{Universidad de Granada, Facultad de Ciencias.
Departamento de An\'{a}lisis Matem\'{a}tico, 18071-Granada
(Spain)}
\email{rchiclana@ugr.es; mmartins@ugr.es; abrahamrueda@ugr.es}
\urladdr[(Mart\'in)]{\url{http://www.ugr.es/~mmartins/}}
\urladdr[(Rueda Zoca)]{\url{https://arzenglish.wordpress.com}}

\thanks{The research of L.\ Garc\'ia-Lirola was supported by the grants Spanish MTM2017-83262-C2-2-P and Fundaci\'on S\'eneca CARM 19368/PI/14, and also by a postdoctoral grant in the framework of \emph{Programa
Regional de Talento Investigador y su Empleabilidad} from \emph{Fundaci\'on S\'eneca - Agencia de Ciencia y Tecnolog\'ia de la Regi\'on de Murcia}. The research of R.\ Chiclana and M.\ Mart\'{\i}n was supported by projects PGC2018-093794-B-I00 (MCIU/AEI/FEDER, UE) and FQM-185 (Junta de Andaluc\'{\i}a/FEDER, UE). The research of Abraham Rueda Zoca was supported by Vicerrectorado de Investigaci\'on y Transferencia de la Universidad de Granada in the program ``Contratos puente'', by MICINN (Spain) Grant
PGC2018-093794-B-I00 (MCIU, AEI, FEDER, UE), by Junta de Andaluc\'ia
Grant A-FQM-484-UGR18 and by Junta de Andaluc\'ia Grant FQM-0185.}

\begin{document}

\begin{abstract}
We study the density of the set $\SA(M,Y)$ of those Lipschitz maps from a (complete pointed) metric space $M$ to a Banach space $Y$ which strongly attain their norm (i.e.\ the supremum defining the Lipschitz norm is actually a maximum). We present new and somehow counterintuitive examples, and we give some applications. First, we show that $\SA(\mathbb T,Y)$ is not dense in $\Lip(\mathbb T,Y)$ for any Banach space $Y$, where $\mathbb T$ denotes the unit circle in the Euclidean plane. This provides the first example of a Gromov concave metric space (i.e.\ every molecule is a strongly exposed point of the unit ball of the Lipschitz-free space) for which the density does not hold. Next, we construct metric spaces $M$ satisfying that $\SA(M,Y)$ is dense in $\Lip(M,Y)$ regardless $Y$ but which contain isometric copies of $[0,1]$ and so the Lipschitz-free space $\mathcal F(M)$ fails the Radon--Nikod\'{y}m property, answering in the negative a question posed in \cite{ccgmr}. Furthermore, an example of such $M$ can be produced failing all the previously known sufficient conditions for the density of strongly norm attaining Lipschitz maps. Finally, among other applications, we prove that given a compact metric space $M$ which does not contain any isometric copy of $[0,1]$ and a Banach space $Y$, if $\SA(M,Y)$ is dense, then $\SA(M,Y)$ actually contains an open dense subset and $B_{\mathcal F(M)}=\overline{\co}(\strexp{B_{\mathcal F(M)}})$. Further, we show that if $M$ is a boundedly compact metric space for which $\SA(M,\R)$ is dense in $\Lip(M,\R)$, then the unit ball of the Lipschitz-free space on $M$ is the closed convex hull of its strongly exposed points.
\end{abstract}

	

\maketitle

\section{Introduction}

In this paper we will solve some questions related to when the set of those Lipschitz maps which strongly attain their norm is dense. Let us start with necessary definitions. A \emph{pointed metric space} is just a metric space $M$ in which we distinguish an element, called $0$. Throughout the paper, the metric spaces will be complete and the Banach spaces will be over the real scalars. Given a pointed metric space $M$ and a Banach space $Y$, we write $\Lip(M,Y)$ to denote the Banach space of all Lipschitz maps $F:M\longrightarrow Y$ which vanish at $0$, endowed with the Lipschitz norm defined by
\begin{equation}\tag{\ensuremath{\spadesuit}}\label{eq:def-norma-Lipschitz}
 \| F \|_L := \sup\left\{\frac{\|F(x)-F(y)\|}{d(x,y)} \colon x,y\in M,\, x \neq y \right\}.
\end{equation}
Let us comment that the choice of the distinguished element is not important, as the resulting spaces of Lipschitz maps are isometrically isomorphic. Following \cite{kms} and \cite{Godefroy-survey-2015}, we say that $F\in \Lip(M,Y)$ \emph{attains its norm in the strong sense} or \emph{strongly attains its norm}, whenever the supremum in \eqref{eq:def-norma-Lipschitz} is actually a maximum, that is, whenever there are $x,y\in M$, $x\neq y$, such that
$$
\frac{\|F(x)-F(y)\|}{d(x,y)}=\|F\|_L.
$$
The set of all Lipschitz maps in $\Lip(M,Y)$ which attain their norm in the strong sense is denoted by $\SA(M,Y)$.

As the starting point of the study of strong norm attainment we can consider the papers \cite{Godefroy-survey-2015,kms}, where the first examples of spaces failing and satisfying that the set of strongly norm attaining functionals is dense are found. On the one hand, the first negative example is \cite[Example 2.1]{kms}, where a norm-one Lipschitz function $f\in \Lip([0,1],\mathbb R)$ is found so that $d(f,\SA([0,1],\R))>0$ (see \cite[p.~109]{Godefroy-survey-2015} for a quantitative sharper result). On the other hand, the first positive result comes from \cite{Godefroy-survey-2015}, where it is proved that $\SA(M,Y)$ is dense in $\Lip(M,Y)$ whenever $M$ is a compact metric space such that the space of little Lipschitz functions uniformly separates points and $Y$ is a finite-dimensional Banach space, leaving as an open question to determine those compact metric spaces $M$ and those Banach spaces $Y$ for which $\SA(M,Y)$ is dense in $\Lip(M,Y)$ \cite[Question 6.7]{Godefroy-survey-2015}.

After that, new positive results were obtained in \cite[Section 4]{lpr} and in \cite[Section 7]{lppr}, where the result \cite[Proposition 7.4]{lppr} is of particular interest. In order to state this result, we need to introduce the Lipschitz-free space. Let $M$ be a pointed metric space. We denote by $\delta$ the canonical isometric embedding of $M$ into $\Lip(M,\R)^*$, which is given by $\langle f, \delta(x) \rangle =f(x)$ for $x \in M$ and $f \in \Lip(M,\R)$. We denote by $\mathcal{F}(M)$ the norm-closed linear span of $\delta(M)$ in the dual space $\Lip(M,\R)^*$, which is usually called the \emph{Lipschitz-free space over $M$}, see the papers \cite{Godefroy-survey-2015} and \cite{gk}, and the book \cite{wea5} (where it receives the name of Arens-Eells space) for background on this. It is well known that $\mathcal{F}(M)$ is an isometric predual of the space $\Lip(M,\R)$ \cite[p.~91]{Godefroy-survey-2015}. With this notion in mind, then \cite[Proposition 7.4]{lppr} asserts that $\SA(M,Y)$ is dense in $\Lip(M,Y)$ for every Banach space $Y$ whenever $\mathcal F(M)$  has the Radon--Nikod\'{y}m property (RNP in short). Later on, taking advantage of the recent progress on the study of the extremal structure of the unit ball of Lipschitz-free spaces \cite{ag,lppr,gprdauga}, an intensive study of strongly norm attaining Lipschitz maps was done in \cite{ccgmr}. Also, a Bishop--Phelps--Bollob\'{a}s type property for Lipschitz maps (this is a stronger quantitative way in which the set of strongly norm attaining Lipschitz maps can be dense) is considered in \cite{cm}. Let us summarise these and some other known results on the density of strongly norm attaining Lipschitz maps; see subsection~\ref{subsec:notation} to find the needed definitions.
\begin{enumerate}
\item[(N1)] If $M$ is a length metric space (in particular, if it is geodesic), then $\SA(M,\mathbb R)$ is not dense in $\Lip(M,\mathbb R)$ \cite[Theorem 2.2]{ccgmr}.
\item[(N2)] If $M$ is a closed subset of an $\R$-tree (in particular, a subset of $\mathbb R$) with positive length measure, then $\SA(M,\mathbb R)$ is not dense in $\Lip(M,\mathbb R)$  \cite[Theorem 2.3]{ccgmr}.
\item[(P1)] If $\Free(M)$ has the RNP, then $\SA(M,Y)$ is dense in $\Lip(M,Y)$ for every Banach space $Y$ \cite[Proposition 7.4]{lppr}.
\item[(P2)] If $M$ satisfies that $B_{\mathcal F(M)}$ is the closed convex hull of a uniformly strongly exposed set, then $\SA(M,Y)$ is dense in $\Lip(M,Y)$ for every Banach space $Y$ \cite[Proposition 3.3]{ccgmr}. In particular, this happens in the following situations:
    \begin{enumerate}
    \item[(P2.1)] if $\mathcal F(M)$ has property $\alpha$ \cite[Corollary 3.10]{ccgmr},
    \item[(P2.2)]\label{equa:Holder} if $M$ is any H\"older space \cite[Proposition 3.8]{cm}.
    \end{enumerate}
\item[(P3)] If $\mathcal F(M)$ has property quasi-$\alpha$, then $\SA(M,Y)$ is dense in $\Lip(M,Y)$ for every Banach space $Y$ \cite[Proposition 3.19]{ccgmr}.
\end{enumerate}

Let us comment on some properties that all these examples have in common. In all the negative results, the space $\Free(M)$ contains isomorphic copies of $L_1[0,1]$. In all the positive results, the unit ball of $\Free(M)$ is the closed convex hull of its strongly exposed points and, actually, it is not known whether in all positive known results, $\F{M}$ has the RNP. In view of all these, the following questions are natural on a complete metric space $M$:
\begin{enumerate}
\item[(Q1)] Does $\mathcal F(M)$ have the RNP if $\SA(M, Y)$ is dense in $\Lip(M, Y)$ for every Banach space $Y$?
\item[(Q2)]  Is $\SA(M,\mathbb R)$ dense in $\Lip(M,\mathbb R)$ if $B_{\mathcal F(M)}=\overline{\co}\bigl(\strexp{B_{\mathcal F(M)}}\bigr)$ (or, even, if $M$ is Gromov concave)?
\item[(Q3)] Conversely, is $B_{\mathcal F(M)}=\overline{\co}\bigl(\strexp{B_{\mathcal F(M)}}\bigr)$ when $\SA(M,\mathbb R)$ is dense in $\Lip(M,\mathbb R)$?
\item[(Q4)] Does $\SA(M,\R)$ fail to be dense in $\Lip(M,\R)$ provided $\mathcal F(M)$ contains an isomorphic (or even isometric) copy of $L_1[0,1]$?
\end{enumerate}

Note that all the questions above have positive answer if $M$ is a compact subset of $\mathbb R$: see \cite[Theorem~3.2]{godard} for a description of $\mathcal F(M)$ as an $L_1$ space and \cite[Corollary 2.6]{ccgmr} for description of those compact subsets of $\R$ for which strongly norm attaining Lipschitz maps are dense.

Let us also notice that (Q1) was asked in \cite[p.~29]{ccgmr}. Also, in the case that $M$ is compact, it is asked in \cite[p.~115]{Godefroy-survey-2015} whether $\mathcal F(M)$ is isometric to the dual of the space of little Lipschitz functions when $\SA(M,\R)$ is dense in $\Lip(M,\R)$. Outside the compact setting, it is known that there are metric spaces $M$ for which $\F{M}$ has the RNP but it is not isometric to any dual space \cite[Example 5.8]{lppr}.

A natural candidate to study is the unit circle $\mathbb T$, as $\mathcal F(\mathbb T)$ does not have the RNP (it contains an isomorphic copy of $L_1[0,1]$) but the curvature of $\T$ suggests that $\mathcal F(\mathbb T)$ should contain a lot of strongly exposed points, according to the characterisation of these points given in \cite[Theorem 5.4]{gprdauga}. Then, depending on whether $\SA(\T,\R)$ is dense in $\Lip(\T,\R)$ or not, it would provide a negative answer to either (Q1) or (Q2). One of the main results of the paper is Theorem \ref{theo:toro}, where we prove that $\SA(\mathbb T,\mathbb R)$ is not dense in $\Lip(\mathbb T,\mathbb R)$, while every molecule (see Subsection \ref{subsec:notation} for the definition) is a strongly exposed point, so $B_{\F{\T}}=\overline{\co}\bigl(\strexp{B_{\mathcal F(\T)}}\bigr)$, which provides a negative answer for (Q2). Let us observe that this example is somehow counterintuitive, as all the previously known negative examples either live in the real line or have some kind of convexity (that is, either are contained in a segment or contain many ``almost'' segments).

For a counterexample for (Q1), we find in Theorem \ref{theo:lluvia} two metric spaces for which strongly norm attaining maps are dense regardless of the range space, while each of them contains an isometric copy of $[0,1]$. This proves that the answer to both (Q1) and (Q4) is negative. In particular, this result seems to provide the first example of a compact metric space $M$ for which $\SA(M,Y)$ is dense in $\Lip(M,Y)$ and whose space of little Lipschitz functions does not separate the points of $M$, answering by the negative the already mentioned question of G.~Godefroy (see the paragraph following \cite[Question 6.7]{Godefroy-survey-2015}).
Again, this example is somehow counterintuitive, as it implies that every function in $\Lip([0,1],\R)$ which is far from $\SA([0,1],\R)$ (and there are many of them) admits extensions to a function in $\Lip(M,\R)$ which can be approximated by strongly norm attaining ones. Furthermore, as another consequence of Theorem \ref{theo:lluvia}, we find in Example~\ref{example:frankenstein} a metric space $M$ for which $\SA(M,Y)$ is dense in $\Lip(M,Y)$ regardless of $Y$, but failing all the sufficient conditions studied in \cite{ccgmr,cm,Godefroy-survey-2015}, which makes of $M$ an example of a metric space in which the density is obtained differently from all the previously known ways.

With respect to (Q3), we begin Section \ref{section:abiden} by showing in Theorem \ref{theorem:closedconvexhullextremepoints} that, given a complete metric space $M$, if $\SA(M,\mathbb R)$ is dense in $\Lip(M,\mathbb R)$, then $$B_{\mathcal F(M)}=\overline{\co}\left(\ext{B_{\mathcal F(M)}}\right).$$ A complete affirmative answer to (Q3) in the compact case is obtained in Theorem~\ref{theorem:clcvstrepcompact} and extended to the boundedly compact case in Corollary \ref{coro:clcvstrepbouncompact}.

Furthermore, in the case that $M$ is a compact metric space which does not contain any isometric copy of $[0,1]$, we even obtain in Theorem~\ref{theo:centraestruct} that $\SA(M,Y)$ actually contains an open dense subset: the one of non-local Lipschitz maps (which actually attain their norms at strongly exposed points, see Lemma \ref{lemma:nolocalstrexp}).
In general, given Banach spaces $X$ and $Y$, the presence of open subsets in the set of norm attaining operators from $X$ to $Y$ is a rare phenomenon (see Remark \ref{remark:abidensoraro}). However, an adaptation of the techniques of Theorem~\ref{theo:centraestruct} allows us to prove in Proposition \ref{prop:denseopen-locallycompact} that if $M$ is a locally compact metric space not containing any isometric copy of $[0,1]$, then the set of absolutely strongly exposing operators from $\mathcal F(M)$ to $Y$ is dense in $\mathcal{L}(\mathcal F(M),Y)$ if, and only if, $\SA(M,Y)$ contains a certain open dense subset $B$ (see Proposition \ref{prop:denseopen-locallycompact} for a description of such set). Notice that, in general, in order to ensure that a certain set of norm attaining operators $\NA(X,Y)$ contains an open subset it is not sufficient that $X$ has the RNP (see Remark \ref{remark:abidensoraro}). However, combining Theorem \ref{theo:centraestruct} and Proposition \ref{prop:denseopen-locallycompact} we obtain that $\SA(M,Y)$ (and henceforth $\NA(\mathcal F(M),Y)$) contains an open dense subset in the classical examples in which $\mathcal F(M)$ is known to enjoy the RNP as in the case when $M$ is uniformly discrete or for the class of compact metric spaces described in \cite[p.~110]{Godefroy-survey-2015}.

Finally, as a by-product of our study, we prove that for all the known sufficient conditions for Lindenstrauss property A, one actually obtains that the absolutely strongly exposing operators form a dense subset (see Section \ref{sect:byproduct}).

\subsection{Notation and a little background}\label{subsec:notation}
We will only consider real Banach spaces. Given a Banach space $X$ we will denote by $B_X$ and $S_X$ the closed unit ball and the closed unit sphere. Also, $X^*$ stands for the topological dual of $X$ and $J_X\colon X\longrightarrow X^{**}$ is the canonical inclusion. A \emph{slice} of the unit ball $B_X$ is a non-empty intersection of an open half-space with $B_X$; every slice can be written in the form \[
S(B_X,f,\beta):=\{x\in B_X \colon f(x)>1-\beta\},
\]
where $f \in S_{X^*}$, $\beta>0$.

The notations $\ext{B_X}$, $\preext{B_X}$, $\strexp{B_X}$ stand for the set of extreme points, preserved extreme points (i.e.\ extreme points which remain extreme in the bidual ball), and strongly exposed points of $B_X$, respectively. A point $x \in B_X$ is said to be a \emph{denting point} of $B_X$ if there exist slices of $B_X$ containing $x$ of arbitrarily small diameter. We will denote by $\dent{B_X}$ the set of denting points of $B_X$. We always have that
$$
\strexp{B_X}\subset \dent{B_X} \subset \preext{B_X} \subset \ext{B_X}.
$$

If $X$ and $Y$ are Banach spaces, we write $\mathcal{L}(X,Y)$ to denote the Banach space of all bounded linear operators from $X$ to $Y$, endowed with the operator norm. We say that $T\in \mathcal{L}(X,Y)$ \emph{attains its norm}, and write $T\in \NA(X,Y)$, if there is $x\in X$ with $\|x\|=1$ such that $\|Tx\|=\|T\|$. The study of the density of norm attaining linear operators has its root in the classical Bishop--Phelps theorem, which states that $\NA(X,\R)$ is dense in $X^*=\mathcal{L}(X,\R)$ for every Banach space $X$. J.~Lindenstrauss extended such study to general linear operators, showing that this is not always possible, and also giving positive results. A Banach space $X$ is said to have \emph{Lindenstrauss property A} when $\overline{\NA(X,Y)}=\mathcal{L}(X,Y)$ for every Banach space $Y$; it is shown in \cite{linds63} that reflexive spaces have this property. This result was extended by J.~Bourgain \cite{bou77} to Banach spaces with the RNP. In order to be more precise, we will introduce a bit of notation. According to \cite{bou77}, given two Banach spaces $X$ and $Y$, an operator $T\colon X\longrightarrow Y$ is \emph{absolutely strongly exposing} if there exists $x \in S_{X}$ such that for every sequence $\{x_n\} \subset B_X$ such that $\lim_{n} \|Tx_n\| = \|T\|$, there is a subsequence $\{x_{n_k}\}$ which converges to either $x$ or $-x$. Clearly, if $T$ is an absolutely strongly exposing operator, then $T$ attains its norm at the point $x$ appearing at the definition and it is easy to show that such point $x\in S_X$ is a strongly exposed point (see Proposition~\ref{prop:strongoperator} for details). The announced result of J.~Bourgain \cite[Theorem~5]{bou77} says that if $X$ is a Banach space with the RNP and $Y$ is any Banach space, then the set of absolutely strongly exposing operators from $X$ to $Y$ is a $G_\delta$-dense subset of $\mathcal{L}(X,Y)$.

In order to connect the theory of norm attaining operators and the theory of strong norm attainment of Lipschitz maps, let us recall that when $M$ is a pointed metric space and $Y$ is a Banach space, it is well known that every Lipschitz map $f \colon M \longrightarrow Y$ can be isometrically identified with the continuous linear operator $\hat{f} \colon \mathcal{F}(M) \longrightarrow Y$ defined by $\hat{f}(\delta_p)=f(p)$ for every $p \in M$. This mapping completely identifies the spaces $\Lip(M,Y)$ and $\mathcal{L}(\mathcal{F}(M),Y)$. Bearing this fact in mind, the set $\SA(M,Y)$ is identified with the set of those elements of $\mathcal{L}(\mathcal{F}(M),Y)$ which attain their operator norm at some \emph{molecule}, that is, at an element of $\mathcal{F}(M)$ of the form
\[
m_{x,y}:=\frac{\delta(x)-\delta(y)}{d(x,y)}
\]
for $x, y \in M$, $x \neq y$. We write $\Mol{M}$ to denote the set of all molecules of $M$. Note that, since $\Mol{M}$ is balanced and norming for $\Lip(M,\mathbb{R})$, a straightforward application of Hahn-Banach theorem implies that
$$
\overline{\co}(\Mol{M})=B_{\mathcal F(M)}.
$$
It is clear now that when $\SA(M,Y)$ is dense in $\Lip(M,Y)$, then $\NA(\mathcal{F}(M),Y)$ has to be dense in $\mathcal{L}(\mathcal{F}(M),Y)$ a fortiori. The converse result is not true as, for instance, $\NA(\mathcal{F}(M),\R)$ is always dense by the Bishop-Phelps theorem but, as we have already mentioned, there are many metric spaces $M$ such that $\SA(M,\R)$ is not dense in $\Lip(M,\R)$. Of course, if $\SA(M,Y)$ is dense in $\Lip(M,Y)$ for every Banach space $Y$, then $\mathcal{F}(M)$ has Lindenstrauss property A. Let us recall the exact definition of some sufficient conditions to get that $\SA(M,Y)$ is dense in $\Lip(M,Y)$ for every $Y$, considered in \cite{ccgmr}, which have been commented in the introduction.

\begin{definition}\label{def:suff-conditions} Let $X$ be a Banach space.
\begin{enumerate}
\item A subset $S\subset S_X$ is said to be a \emph{uniformly strongly exposed set} (or a \emph{set of uniformly strongly exposed points}) \cite{linds63} if there is a family of functionals $\{h_x\}_{x\in S}$ with $\|h_x\|=h_x(x)=1$ for every $x\in S$ such that, given $\varepsilon>0$ there is $\delta>0$ satisfying that
\[
\sup_{x\in S} \diam(S(B_X,h_x, \delta)) \leq \varepsilon;
\]
equivalently, if for every $\varepsilon>0$ there is $\delta'>0$ such that whenever $z\in B_X$ satisfies $h_x(z)>1-\delta'$ for some $x\in S$, then $\|x-z\|<\varepsilon$ (that is, all elements of $S$ are strongly exposed points with the same relation $\varepsilon$--$\delta$).
\item $X$ has \emph{property $\alpha$} \cite{schachermayer} if there exist a balanced subset $\{x_\lambda\}_{\lambda \in \Lambda}$ of $X$ and a subset $\{x^*_\lambda\}_{\lambda \in \Lambda} \subseteq X^*$ such that
	\begin{enumerate}
		\item[(i)] $\lVert x_\lambda \rVert = \lVert x^*_\lambda\rVert = \lvert x^*_\lambda(x_\lambda)\rvert =1$ for all $\lambda\in \Lambda$.
		\item[(ii)] There exists $0\leq \rho <1$ such that
		\[ |x^*_\lambda(x_\mu)|\leq \rho \quad \forall \, x_\lambda \neq \pm x_\mu. \]
		\item[(iii)] $\overline{\co}\left(\{x_\lambda\}_{\lambda \in \Lambda}\right)= B_X$.
	\end{enumerate}	
\item $X$ has \emph{property quasi-$\alpha$} \cite{ChoiSong} if there exist a balanced subset $A:=\{x_\lambda\colon \lambda \in \Lambda\}$ of $X$, a subset $\{x_\lambda^*\colon \lambda \in \Lambda\}\subseteq X^*$, and $\rho \colon \Lambda \longrightarrow \mathbb{R}$ such that
	\begin{itemize}
		\item[(i)] $\| x_\lambda \|= \| x^*_\lambda\| =| x^*_\lambda (x_\lambda)|=1 $ for all $\lambda \in \Lambda$.
		\item[(ii)] $|x^*_\lambda(x_\mu)| \leq \rho(\lambda)<1$ for all $x_\lambda \neq \pm x_\mu$.
		\item[(iii)] For every $e \in \ext{B_{X^{**}}}$, there exists a subset $A_e \subseteq A$ such that either $e$ or $-e$ belong to $\overline{J_X(A_e)}^{\,\omega^*}$ and $r_e=\sup\{\rho(\mu)\colon x_\mu \in A_e\}<1$.
	\end{itemize}
\end{enumerate}
\end{definition}

It follows that, in the definition of property quasi-$\alpha$ (and henceforth in property $\alpha$), every point $x_\lambda$ is a strongly exposed point.

Given a metric space $M$, $B(x,r)$ denotes the closed ball in $M$ centered at $x \in M$ with radius $r$. The space $M$ is said to be \emph{boundedly compact} if every closed ball is compact. Given $x, y \in M$, we write $[x,y]$ to denote the \emph{metric segment} between $x$ and $y$, that is,
\[
[x,y] := \{z \in M \colon d(x,z)+d(z,y)=d(x,y)\}.
\]
Given $x,y,z\in M$, the \emph{Gromov product of $x$ and $y$ at $z$} is defined as
$$(x,y)_z:=\frac{1}{2}(d(x,z)+d(y,z)-d(x,y)).$$
Related to the definition of Gromov product is the definition of property (Z). Given $x,y\in M$ with $x\neq y$, we say that the pair $(x,y)$ has \emph{property (Z)} if, for every $\varepsilon>0$, there exists $z\in M\setminus\{x,y\}$ satisfying that
$$(x,y)_z<\varepsilon \min\{d(x,z),d(y,z)\}.$$
It is known that the pair $(x,y)$ fails property (Z) if, and only if, the molecule $m_{x,y}$ is strongly exposed \cite[Theorem 5.4]{gprdauga}.

According to \cite{cm}, a metric space $M$ is said to be \emph{Gromov concave} if, for every pair of distinct points $x,y\in M$, there exists $\varepsilon_{x,y}>0$ such that
$$(x,y)_z\geq \varepsilon_{x,y}\min\{d(x,z),d(y,z)\}$$
holds for every $z\in M\setminus\{x,y\}$. By the above paragraph, this is equivalent to the fact that all molecules are strongly exposed points of the unit ball of $\Free(M)$.

Connected with property (Z) is the concept of \emph{length} and \emph{geodesic} metric space. Given a metric space $M$, we say that \emph{$M$ is a length space} if $d(x,y)$ is equal to the infimum of the length of the rectifiable curves joining $x$ and $y$ for every pair of points $x,y\in M$. In the case that such an infimum is actually a minimum, it is said that $M$ is \textit{a geodesic space}. It is clear that every geodesic space is a length space, but Example 2.4 in \cite{ikw} shows that the converse is not true. On the other hand, length spaces have been recently considered in \cite{gprdauga}, where it is proved that a metric space $M$ is length if, and only if, $\Lip(M,\R)$ has the Daugavet property \cite[Theorem 3.5]{gprdauga}. Note by passing that for a complete metric space $M$, it is known that if $M$ is length then every pair of different points $x,y\in M$ enjoys property (Z) \cite[Proposition 2.3 and Proposition 2.8]{ikw}. The converse has been recently proved in \cite[Main Theorem]{am}. Also, it was proved in \cite[Proposition 3.4]{gprdauga} that a complete metric space $M$ is length if, and only if, every Lipschitz function of $\Lip(M,\mathbb R)$ is \emph{local}. Given a Lipschitz map $f:M\longrightarrow Y$, we say that $f$ is \emph{local} if, for every $\varepsilon>0$, there are $x,y\in M$ such that $0<d(x,y)<\varepsilon$ and $\Vert \hat{f}(m_{x,y})\Vert>\Vert f\Vert_L-\varepsilon$. This means that $f$ approximates its norm at arbitrarily close points. So, the Lipschitz functions which are farthest to the local ones are the elements of the \emph{little Lipschitz space}, $\lip(M,\R)$, that is, the space of those $f\in \Lip(M,\R)$ such that for any $\varepsilon>0$ there exists $\delta>0$ such that if $d(x, y)<\delta$, then $|f(x)-f(y)| \leqslant \varepsilon d(x, y)$. On the other hand, let us comment that since every strongly exposed point of $B_{\F{M}}$ is a molecule \cite{wea5}, it is clear that if $\hat{f}\in \F{M}^*$ is a strongly exposing functional, then $f$ is a non-local Lipschitz function by \cite[Lemma 1.3]{ccgmr}. A partial converse also holds: if $M$ is compact and $f\in \Lip(M,Y)$ is non-local, then $\hat{f}$ attains its norm at a strongly exposed molecule (see Lemma \ref{lemma:nolocalstrexp}) and, in particular, $f\in \SA(M,Y)$.

\section{The new examples}
It is well known that there is a close relation between Lindenstrauss property A and the presence of a rich extremal structure in a Banach space. For instance, it is a classical result from \cite[Theorem 2]{linds63} that, given a Banach space $X$ which admits an equivalent locally uniformly rotund renorming (in particular, a separable one), if $X$ has Lindenstrauss property A, then $B_X=\overline{\co}\bigl(\strexp{B_X}\bigr)$. From this fact, it is immediate, for instance, that if $M=[0,1]$ then $\NA(\mathcal F(M),Y)$ is not dense in $\Lip(M,Y)$ for some Banach space $Y$, since $\mathcal F(M)=L_1[0,1]$ (c.f.\ e.g.\  \cite[Example 2.1]{Godefroy-survey-2015}) and then, $B_{\F{M}}$ has no extreme points. Also, as we already commented in the introduction, it is not difficult to prove that $\SA([0,1],\mathbb R)$ is not dense in $\Lip([0,1],\mathbb R)$ because any strongly norm attaining Lipschitz function on $[0,1]$ is affine in a whole segment as soon as it attains its norm at its extreme points \cite[Lemma 2.2]{kms}.

The case of the unit sphere of the Euclidean plane $\mathbb T$ is, however, quite more delicate. On the one hand, $\mathbb T$ contains subsets which are bi-Lipschitz equivalent to a segment in $\mathbb R$ (in particular, $\mathcal F(\mathbb T)$ contains isomorphic copies of $L_1[0,1]$, so it fails the RNP), which should make it difficult for $\SA(\mathbb T,\mathbb R)$ to be dense in $\Lip(\mathbb T,\mathbb R)$. On the other hand, the curvature of $\mathbb T$ suggests an abundance of strongly exposed points in $B_{\mathcal F(\mathbb T)}$ thanks to \cite[Theorem 5.4]{gprdauga}, which could help to get density of $\SA(\mathbb T,\mathbb R)$ (for instance, such density would be obtained if $B_{\mathcal F(\mathbb T)}$ were the closed convex hull of a \emph{uniformly} strongly exposing set according to \cite[Proposition 3.3]{ccgmr}). This fact makes of $\mathbb T$ an interesting example to analyse because, if $\SA(\mathbb T,\mathbb R)$ were dense, we would get a negative answer to (Q1); if not, $\mathbb T$ would be a counterexample to question (Q2). This is what is done in the following theorem.

\begin{theorem}\label{theo:toro} Let $\mathbb T$ be the unit sphere of the Euclidean plane endowed with the inherited Euclidean metric. Then:
\begin{enumerate}
  \item[(a)] $\overline{\SA(\mathbb{T},\mathbb{R})}\neq \Lip(\mathbb{T},\mathbb{R})$,
  \item[(b)] but $\mathbb{T}$ is Gromov concave, that is, $m_{x,y}\in \strexp{B_{\Free(\mathbb T)}}$ for every pair of distinct points $x,y\in \mathbb T$.
\end{enumerate}
\end{theorem}

Prior to its proof, let us comment some interesting remarks.

\begin{remark}
\begin{enumerate}
  \item[(a)] {\slshape As far as we know, $\T$ is the first known example of a Gromov concave metric space for which strongly norm attaining functionals are not dense. This provides a negative answer to question (Q2).}
  \item[(b)] {\slshape The Banach space $\Free(\T)$ satisfies that its unit ball is the closed convex hull of the set of its strongly exposed points, but strongly exposing functionals are not dense in $\Free(\T)^*$.}\\ Indeed, the first assertion is given by (b) in the theorem above, while the second one follows from (a) and the fact that strongly exposing functionals attain their norm on strongly exposed points, and strongly exposed points of a Lipschitz-free space are molecules.
\end{enumerate}
\end{remark}

Let us now prove Theorem \ref{theo:toro}. In order to prove assertion (a), we will need the following key result, which has been suggested to us by F.~Nazarov.

\begin{lemma}\label{lema:nazarov}
Let $M=([0,1],d)$ where $d(x,y):=\vert \e^{ix}-\e^{iy}\vert=\sqrt{2(1-\cos( x-y))}$. Then, there exists a compact subset $C$ inside the open interval $]0,1[$ such that the function $f\in \Lip(M,\mathbb{R})$ defined by
$$
f(x):=\int_0^x \chi_C (t)\ dt \quad \forall \, x \in [0,1]
$$
is a norm-one Lipschitz function which does not belong to $\SA(M,\mathbb{R})$.
\end{lemma}

\begin{proof}
Consider a Cantor set $C=\bigcap_{n=0}^\infty C_n$, where $C_0=[1/4,3/4]$ and $C_{n+1}$ is obtained by removing an interval of length $\lambda(I)^2$ at the middle of each connected component $I$ of $C_{n}$. Note that $C_n$ has $2^{n}$ connected components, all of them with the same length $\frac{\lambda(C_n)}{2^n}$. By construction, $\lambda(C_n\setminus C_{n+1})=2^n \left( \frac{\lambda(C_n)}{2^n}\right)^2$. Taking into account that $\lambda(C_n)<\frac{1}{4}$ for $n\geq 1$, it follows that
\[ \lambda(C) = \frac{1}{2}- \sum_{n=0}^\infty \lambda(C_n\setminus C_{n+1}) = \frac{1}{2}-\sum_{n=0}^\infty 2^n \left(\frac{\lambda(C_n)}{2^n}\right)^2 > \frac{1}{2} - \sum_{n=0}^\infty \frac{1}{4}\frac{1}{2^n} =0.\]
Consider the Lipschitz function $f\colon ([0,1],d)\longrightarrow \mathbb R$ given by $\displaystyle f(x)=\int_0^x \chi_C (t)\, dt$ for every $x\in [0,1]$. Note that $\norm{f}_L \neq 0$ since $\lambda(C)>0$. We claim that $f$ does not attain its Lipschitz norm. Indeed, assume that there are $x,y \in [0,1]$, $x<y$, such that
\[ \norm{f}_L=\frac{|f(y)-f(x)|}{d(x,y)} = \frac{f(y)-f(x)}{d(x,y)}.\]
Clearly, $x,y\in [1/4,3/4]$.  We claim that $x, y\in C$. Indeed, assume that $x\notin C$. Then there is $0<\varepsilon<y-x$ such that $(x,x+\varepsilon)\cap C=\emptyset$. Thus, $f(x)=f(x+\varepsilon)$. Then,
\[ \norm{f}_L = \frac{f(y)-f(x)}{d(x,y)}<\frac{f(y)-f(x+\varepsilon)}{d(x+\varepsilon, y)},\]
a contradiction. So $x\in C$. Analogously, we get that $y\in C$.
Now, let $n$ be the maximum integer such that $x$ and $y$ belong to the same connected component $I$ of $C_n$. Since $x$ and $y$ do not belong to the same connected component of $C_{n+1}$, there are $u,v$ such that $(u,v)\subset C_{n}\setminus C_{n+1}$ and $|u-v|= \lambda(I)^2\geq |x-y|^2$. Note also that $(u,v)\cap C=\emptyset$ and so $f(u)=f(v)$.  We have
\begin{eqnarray*}
\norm{f}_L = \frac{f(y)-f(x)}{d(x,y)} = \frac{f(y)-f(v)+f(u)-f(x)}{d(x,y)}\leq \norm{f}_L \frac{d(y,v)+d(u,x)}{d(x,y)}
\end{eqnarray*}
and so $d(x,y)\leq d(y,v)+d(u,x)$. 
One can check that
\[ t-\frac{t^3}{24} \leq \sqrt{2(1-\cos(t))}\leq t \quad \forall t\in[0,1].\]
Thus,
\begin{align*}
(y-x)- \frac{(y-x)^3}{24} &\leq \sqrt{2(1-\cos(y-x))} = d(x,y) \leq d(y,v)+d(u,x)\\
 &= \sqrt{2(1-\cos(y-v))}+\sqrt{2(1-\cos(u-x))} \\
 &\leq y-v+u-x \leq y-x-(y-x)^2.
\end{align*}
Therefore, $(y-x)^2\leq \frac{(y-x)^3}{24}$, a contradiction. Thus, $f$ does not attain its Lipschitz norm. It remains to show that $\norm{f}_L=1$. First, note that
\[ \norm{f}_L=\sup_{x,y\in[0,1]}\frac{|f(x)-f(y)|}{d(x,y)}\geq\sup_{x,y\in[0,1]}\frac{|f(x)-f(y)|}{|x-y|}=\norm{f'}_\infty =1. \]
Now, pick a pair of sequences $\{x_n\}, \{y_n\}$ with $x_n\neq y_n$ for every $n$ and such that $\frac{f(x_n)-f(y_n)}{d( x_n,y_n)}\longrightarrow \norm{f}_L$. Observe that $d(x_n,y_n)\longrightarrow 0$. Otherwise, we could extract, by compactness, subsequences $\{x_{n_k}\}$ and $\{y_{n_k}\}$ converging to different points $x$,$y$ in $[0,1]$, so $f$ would attains its Lipschitz norm at the pair $(x,y)$, a contradiction. Consequently, $\frac{|y_n-x_n|}{d(x_n,y_n)}\longrightarrow 1$ and then
\[  \lim_{n\to\infty}\hat{f}(m_{x_n,y_n})=\lim_{n\to\infty} \frac{f(x_n)-f(y_n)}{d(x_n,y_n)}= \lim_{n\to\infty} \frac{f(x_n)-f(y_n)}{|y_n-x_n|}\frac{|y_n-x_n|}{d(x_n,y_n)}\leq 1. \]
Hence, we conclude that $\|f\|_L=1$.
\end{proof}

We are now able to show the proof of the theorem.

	\begin{proof}[Proof of Theorem \ref{theo:toro}]
(a). 		Let $A\subseteq \mathbb{T}$ be the following arc of $\mathbb{T}$:
		\[ A=\{\e^{it} \colon t \in [0,1]\}.\]
Let us first show that $\overline{\SA(A,\mathbb{R})}\neq \Lip(A,\mathbb{R})$. In order to do so, remember that $\Lip([0,1],\mathbb{R})$ is isometrically isomorphic to $L_\infty[0,1]$, where the isometry is given by the derivative operator, and observe that $\Phi \colon \Lip(A,\mathbb{R}) \longrightarrow \Lip([0,1],\mathbb{R})$ given by
		\[ [\Phi(f)](t)=f(\e^{it}) \quad \forall\, f \in \Lip(A,\mathbb{R}), \quad \forall \, t \in[0,1]\]
		defines a linear isomorphism. Consequently, a Lipschitz function $g$ will be close to $f$ if, and only if, $\Phi(g)' \in L_\infty[0,1]$ is close to $\Phi(f)'$. Furthermore, we know that there exists a constant $0<K<1$ such that
		\[ K |u-v| \leq |\e^{iu}-\e^{iv}| \leq |u-v| \quad \forall \, u, v \in [0,1].
\]
		Now, let $C$ be the set given by Lemma \ref{lema:nazarov}. We define $f \in \Lip(A,\mathbb{R})$ by
		\[ f(\e^{ix})=\int_{0}^{x} \chi_C(t)\, dt \quad \forall \,t \in [0,1].\]
	 Let us consider $0<\delta< \frac{K}{2}$ and define $h \in \Lip(A,\mathbb{R})$ such that
		\[ \Phi(h)'(x)=  \left\{
		\begin{array}{ll}
		1 & \text{ if $x \in C$}, \\
		-\delta & \text{ if $x \notin C$}. \\
		\end{array} \right. \]
		We will show that if $g \in \Lip(A,\mathbb{R})$ verifies that $\|\Phi(g)' - \Phi(h)'\|_\infty < \delta$ and $\|\Phi(g)\|_L=\|\Phi(h)\|_L=1$, then $g$ does not attain its Lipschitz norm. Firstly, note that if $\|\Phi(g)\|_L=1$ then $\|g\|_L\geq 1$ and $\|\Phi(g)'\|_\infty = 1$, and so
		\[ \Phi(g)'(x) \in  \left\{
		\begin{array}{ll}
		(1-\delta,1) & \text{ if $x \in C$}, \\
		(-2\delta,0) & \text{ if $x \notin C$}. \\
		\end{array} \right. \]
		Let us prove that $\hat{g}(m_{e^{iu},e^{iv}})<1$ for every molecule $m_{\e^{iu},\e^{iv}} \in \Mol{A}$. Let us distinguish two cases:

\noindent \emph{Case $1$:}\  $u<v$. In this case we have that
			\begin{align*}
			g(\e^{iu})-g(\e^{iv}) &= - \int_{u}^{v} \Phi(g)'(t)\, dt \leq \int_{[u,v]\cap C} \delta-1\, dt + \int_{[u,v]\setminus C} 2\delta\, dt\\
			& \leq 2\delta |u-v| < K|u-v|< |\e^{iu}-\e^{iv}|,
			\end{align*}
			so $g$ cannot attain its Lipschitz norm at the molecule $m_{\e^{iu},\e^{iv}}$.

\noindent \emph{Case $2$:}\  $u>v$. In this case we have that 			\begin{align*}
			g(\e^{iu})-g(\e^{iv})&= \int_{v}^{u} \Phi(g)'(t)\, dt \leq \int_{[v,u]\cap C} 1\, dt + \int_{[v,u]\setminus C} 0\, dt\\
			&= \int_{v}^{u} \chi_C(t)\, dt = \int_{v}^{u} \Phi(f)'(t)\, dt\\ &= f(\e^{iu})-f(\e^{iv})<\|f\|_L|\e^{iu}-\e^{iv}|=|\e^{iu}-\e^{iv}|,
			\end{align*}
			since $f$ does not attain its Lipschitz norm by Lemma \ref{lema:nazarov}. Consequently $\|g\|_L=1$ and $g$ does not attain its Lipschitz norm at the molecule $m_{\e^{iu},\e^{iv}}$. By the arbitrariness of $u$ and $v$, we get that $g\notin \SA(A,\mathbb R)$.

		Consequently, $h\notin \overline{\SA(A,\mathbb R)}$. Now let us consider an extension of $h$, say $\varphi$, satisfying that $\varphi\notin \overline{\SA(\mathbb T,\mathbb R)}$. In order to do so, pick $0<\eta<1$ and define
		$$\varphi(\e^{ix}):=\begin{cases}
		    h(\e^{ix}) & x\in [0,1], \\
		    h(\e^{i}) & x\in [1,1+\eta],\\
		    -\frac{x}{2}+h(\e^{i})+\frac{1+\eta}{2} & x\in [1+\eta, 2h(\e^{i})+1+\eta],\\
		    0 & x\in [2h(\e^{i})+1+\eta,2\pi].
		\end{cases}
$$
		It is clear from the definition that $\varphi\in \Lip(\mathbb T,\mathbb{R})$ with $\Vert \varphi\Vert_L=\Vert h\Vert_L=1$ and satisfies that, for every sequence of molecules $\{m_{e^{it_n},e^{is_n}}\}$ such that $\hat{\varphi}(m_{e^{it_n},e^{is_n}})\longrightarrow 1$  there exists a natural number $m$ such that $t_n,s_n\in ]0,1[$ for all
		$n\geq m$. From that fact and the fact that $h\notin \overline{\SA(A,\mathbb R)}$ it follows immediately that $\varphi\notin \overline{\SA(\mathbb T,\mathbb R)}$, as desired.
 		
(b). We have to check that $m_{x,y} \in \strexp{B_{\mathcal F(\mathbb T)}}$ for every $x,y\in \mathbb T$. 	Clearly, we may assume that $y=1$ and $x=\e^{it}$ with $t\in(0, \pi]$ as, clearly, isometries of $\T$ can be used to carry strongly exposed molecules to strongly exposed molecules using the characterization \cite[Theorem 5.4]{gprdauga}. Let us define the continuous function $\phi\colon [-\pi+t/2,t/2]\setminus \{0\} \longrightarrow \mathbb{R}$ given by
	\[ \phi(s) := \frac{1}{8}\left| \frac{\e^{is}-1}{|\e^{is}-1|}-\frac{\e^{it}-1}{|\e^{it}-1|}\right|^2.\]
	A simple calculation shows that $\varepsilon := \inf\bigl\{ \phi(s) \colon s \in [-\pi+t/2,t/2]\setminus \{0\}\bigr\}>0$. We claim that
	\[ (x,y)_z \geq \varepsilon \min\{d(x,z), d(y,z)\} \]
	for every $z\in \mathbb T\setminus\{x,y\}$ and so the pair $(x,y)$ fails property (Z). Indeed, let $z=\e^{is}$. By symmetry, we may assume that $s\in [-\pi+t/2, t/2]\setminus\{0\}$ and so, $$\min\{d(x,z),d(y,z)\} = d(y,z).$$ Now, Clarkson's inequality \cite[Theorem 3]{Clarkson} yields that
	\[ |\e^{it}-1|\leq (1-2\delta(\alpha_1))|\e^{it}- \e^{is}|+(1-2\delta(\alpha_2))|\e^{is}-1|, \]
	where
	\[ \alpha_1 = \left|\frac{\e^{it}- \e^{is}}{|\e^{it}-\e^{is}|}-\frac{\e^{it}-1}{|\e^{it}-1|}\right|, \quad \alpha_2 = \left|\frac{\e^{is}-1}{|\e^{is}-1|}- \frac{\e^{it}-1}{|\e^{it}-1|}\right|,\]
	and $\delta(u)=1-(1-u^2/4)^{1/2}\geq u^2/8$ is the modulus of uniform convexity of $\mathbb R^2$. Thus,
	\begin{align*}
	\frac{(x,y)_z}{d(y,z)} &\geq \frac{\delta(\alpha_1)|\e^{it}-\e^{is}|+\delta(\alpha_2)|\e^{is}-1|}{|\e^{is}-1|}\\
 		    &\geq \frac{1}{8}\alpha_1^2 \frac{|\e^{it}-\e^{is}|}{|\e^{is}-1|} + \frac{1}{8}\alpha_2^2 \geq \frac{1}{8}\alpha_2^2 = \phi(s)\geq \varepsilon,
	\end{align*}
	 as desired.
		\end{proof}

\begin{remark}\label{remark:distequi[0,1]}
We do not know if there exists a distance $d'$ on $[0,1]$, equivalent to the usual one, such that $\SA(([0,1],d'),\R)$ is dense in $\Lip(([0,1],d'),\R)$. Observe that Lemma \ref{lema:nazarov} and the proof of Theorem~\ref{theo:toro} provide a concrete equivalent distance $d$ on $[0,1]$ which makes $([0,1],d)$ Gromov concave and for which  $\SA(([0,1],d),\R)$ is not dense in $\Lip(([0,1],d),\R)$. On the other hand, any H\"{o}lder distance on $[0,1]$ provides the density (as the corresponding Lipschitz-free space has the RNP, see \cite[Corollary 4.39]{wea5} for instance, see also \cite[Corollary 3.7]{cm}). But H\"{o}lder distances are not equivalent to the original ones.
\end{remark}

Note that Theorem \ref{theo:toro} proves that the answer to question (Q2) is negative. In the proof that $\SA(\mathbb T,\mathbb R)$ is not dense in $\Lip(\mathbb T,\mathbb R)$ it is essential the fact that $\mathbb T$ has a ``plenty'' of subsets which are bi-Lipschitz equivalent to intervals in $\mathbb R$. One may think that this property is enough to provide the lack of density of the set of strongly norm attaining Lipschitz maps. However, we are going to show that there are metric spaces $M$ containing copies of $[0,1]$ but which satisfy that $\SA(M,Y)$ is dense in $\Lip(M,Y)$ for every $Y$ (thanks to the fact that they contain dense discrete subsets which provide rich extremal structure to $\mathcal F(M)$).

In the following theorem we construct two metric spaces with the desired properties, proving that the density of $\SA(M,Y)$ for every Banach space $Y$ does not imply the RNP, answering a question from \cite[Section 3.4]{ccgmr} and \cite[p.~115]{Godefroy-survey-2015}, and giving a negative answer to (Q1).
	
\begin{theorem}\label{theo:lluvia}
Consider the subsets of $\R^2$ given by
\begin{align*}
A_n&=\left \{\left (\frac{k}{2^n},\frac{1}{2^n}\right ) \colon k \in \{0,\ldots,2^n\} \right \} \subseteq \mathbb{R}^2 \ \ \forall \, n \in \mathbb{N}\cup \{0\},\\
M_\infty &= \bigcup_{n=0}^\infty A_n, \quad M=M_\infty \cup ([0,1]\times\{0\}).
\end{align*}
Let $\M_p$ be the set $M$ endowed with the distance inherited from $(\R^2,\|\cdot\|_p)$ for $p=1,2$.
Then, $\SA(\M_p,Y)$ is dense in $\Lip(\M_p,Y)$ for every Banach space $Y$ and for $p=1,2$. Moreover,
\begin{enumerate}
\item[(a)] $\Free(\M_1)$ has property $\alpha$,
\item[(b)] The unit sphere of $\Free(\M_2)$ does not contain any uniformly strongly exposed set which generates the ball by closed convex hull.
\end{enumerate}
\end{theorem}

We divide the proof of the theorem into several steps. We start by showing that $\SA(\M_p,Y)$ is dense in $\Lip(\M_p,Y)$ for every $Y$ and $p=1,2$. Actually, we will give a more general result.
	
\begin{proposition}\label{prop:lluvia}
Let $M_\infty$, $M\subseteq \mathbb{R}^2$ be the sets defined in Theorem \ref{theo:lluvia}, and let $|\cdot|$ be a norm in $\R^2$ satisfying that $\norm{\cdot}_\infty\leq |\cdot|\leq\norm{\cdot}_1$. Consider now $\M$ to be the set $M$ endowed with the distance inherited from $(\R^2,|\cdot|)$. Then, $\SA(\M,Y)$ is dense in $\Lip(\M,Y)$ for every Banach space $Y$.
\end{proposition}
	
\begin{proof}
		Let $f \in \Lip(M,Y)$ with $\|f\|_L=1$. Our aim is to approximate $f$ by strongly norm attaining Lipschitz maps, so we may assume that $f$ does not strongly attain its norm. In order to clarify the proof, let us introduce some notation.  For every $n \in \mathbb{N}\cup\{0\}$ and $k \in\{0,\ldots,2^n\}$, we denote by $(n,k)$ the point $\left (\frac{k}{2^n}, \frac{1}{2^n}\right ) \in M_\infty$. Given $n \in \mathbb{N}\cup\{0\}$ and $k \in \{0,\ldots,2^n-1\}$, we write $h_{n,k}$ to denote the molecule $m_{(n,k),(n,k+1)}$. We will say that $$H=\{h_{n,k} \colon n \in \mathbb{N}\cup\{0\}, k \in \{0,\ldots,2^n-1\}\}$$ is the set of horizontal molecules. Given $n \in \mathbb{N} \cup\{0\}$ and $k \in \{0,\ldots,2^n\}$, we write $v_{n,k}$ to denote the molecule $m_{(n,k),(n+1,2k)}$. We will say that
		\[ V=\{v_{n,k} \colon n \in \mathbb{N}\cup\{0\}, k \in\{0,\ldots,2^n\} \}\]
		is the set of vertical molecules. Finally, we define $\Gamma=\pm H \cup \pm V$.
		
		Fix $\varepsilon>0$ and let us distinguish two cases: First of all, assume that
		\[ \rho = \sup\bigl\{\bigl\|\hat{f}(m)\bigr\| \colon m \in \Gamma\bigr\} <1.\]
		Since $M_\infty$ is dense, we may find $u=(n_1,k_1)$, $v=(n_2,k_2) \in M_\infty$ such that $$\frac{k_1}{2^{n_1}}\neq \frac{k_2}{2^{n_2}}, \quad n_1\neq n_2 \quad \text{ and } \quad \hat{f}(m_{u,v})>\frac{1+\rho\varepsilon}{1+\varepsilon}.$$ Let us write $n_3=\max\{n_1,n_2\}$ and consider the set
		\[ N=\bigcup_{n=0}^{n_3} A_n. \]
		Note that if we denote by $\varphi_0$ the restriction of $f$ to $A_{n_3}$, we have $\|\varphi_0\|_L\leq \rho <1$. Then, we may extend this function to a Lipschitz function $\varphi\colon[0,1] \longrightarrow Y$ with $\|\varphi\|_L\leq \rho<1$ (we may define it affine in the gaps). We define $h \colon M_\infty \longrightarrow Y$ by
		\[ h((n,k))=  \left\{
		\begin{array}{ll}
		f((n,k)) & \text{ if $n\leq n_3$; } \\
		\varphi\left (\frac{k}{2^{n}}\right ) & \text{ if $n> n_3$.} \\
		\end{array} \right. \]
		By the way we have extended $\varphi_0$, it is clear that
		\[ \sup\left\{\bigl\|\hat{h}(m)\bigr\|\colon m \in \Gamma\right\}\leq \sup\left\{\bigl\|\hat{f}(m)\bigr\|\colon m \in \Gamma\right\}=\rho<1.\]
		Furthermore,  $|(s,0)| = s \leq |(s,t)|$ and $|(0,t)| = t \leq |(s,t)|$ for every $s$, $t \in \mathbb{R}$. Consequently, if $p$, $q$ are distinct points of $M_\infty\setminus N$, then we may find a molecule $m \in \Gamma$ such that $\|\hat{h}(m_{p,q})\|\leq \|\hat{h}(m)\|$. Indeed, given two different points  $p=(\frac{k_1}{2^n},\frac{1}{2^n})$ and $q=(\frac{k_2}{2^m},\frac{1}{2^m})$ of $M_\infty \setminus N$, we assume with no loss of generality that $n\geq m$, define $q':=(\frac{2^{n-m}k_2}{2^n},\frac{1}{2^n})$. By the assumptions on the norm we get that $\vert p-q'\vert\leq \vert p-q\vert $ and, since $\frac{k_2}{2^m}=\frac{2^{n-m}k_2}{2^n}$, we obtain that
\[
\begin{split}
\left\Vert\hat h(m_{p,q})\right\Vert=\frac{\Vert \varphi(\frac{k_1}{2^n})-\varphi(\frac{k_2}{2^m})\Vert}{\vert p-q\vert}\leq \frac{\Vert \varphi(\frac{k_1}{2^n})-\varphi(\frac{2^{n-m}k_2}{2^n})\Vert}{\vert p-q'\vert}=\Vert \hat h(m_{p,q'})\Vert,
\end{split}
\]		
and notice that $m_{p,q'}\in\co(\Gamma)$. Given $\varepsilon>0$, let us define
\begin{equation}\label{eq:g1}
g\colon M\longrightarrow Y, \qquad g=f+\varepsilon h.
\end{equation}
		It is clear that $\|g-f\|_L\leq \varepsilon$, so it will be enough to show that $g$ strongly attains its norm. On the one hand, note that
		\[ \norm{\hat{g}(m_{u,v})}=\norm{\hat{f}(m_{u,v})+\varepsilon \hat{h}(m_{u,v})}> \frac{1+\rho\varepsilon}{1+\varepsilon}(1+\varepsilon)=1+\rho\varepsilon.\]
		On the other hand, given $p$, $q$ distinct points of $M_\infty\setminus N$, we have that
		\[ \norm{\hat{g}(m_{p,q})}\leq 1+\varepsilon \norm{\hat{h}(m_{p,q})}\leq 1+\varepsilon\rho< \norm{\hat{g}(m_{u,v})}.\]
		Therefore, $g$ cannot approximate its norm at points of $M_\infty\setminus N$. Since $M_\infty\setminus N$ is dense in $[0,1]\times\{0\}$, this implies that $g$ cannot approximate its norm at arbitrarily close points, that is, $g$ is non-local. Consequently, by compactness of $M$, we conclude that $g$ must strongly attain its norm.

		Secondly, assume that $\sup\bigl\{\bigl\|\hat{f}(m)\bigr\| \colon m \in \Gamma\bigr\}=\|f\|_L$. In this case we need to define two kinds of functionals. By a density argument, it will be enough to define them on $M_\infty$. First of all, we will define functionals associated to the vertical molecules. Fix $n \in \mathbb{N}\cup\{0\}$, $k \in \{0,\ldots,2^n\}$. Then, we define $f_{n,k}\colon M_\infty \longrightarrow\mathbb{R}$ given by
		\[ f_{n,k}(p)=  \left\{
		\begin{array}{ll}
		\frac{1}{2^{n+1}} & \text{ if $p=(n,k)$}; \\
		0 & \text{ if $p \neq (n,k)$}. \\
		\end{array} \right. \]
		Note that
		\[ \hat{f}_{n,k}(v_{n,k})=\frac{f_{n,k}((n,k))-f_{n,k}((n+1,2k))}{\vert(n,k)-(n+1,2k)\vert} = \frac{1/2^{n+1}}{1/2^{n+1}}=1.\]
		Furthermore, if $(n',k') \in M$ is such that $m_{(n,k),(n',k')} \in \Gamma$ and $(n',k')\neq (n+1,2k)$, then we have that $\vert (n,k)-(n',k')\vert\geq \frac{3}{2^{n+2}}$, which implies that
		\[ \frac{|f_{n,k}((n,k))-f_{n,k}((n',k'))|}{\vert (n,k)-(n',k')\vert} \leq \frac{1/2^{n+1}}{3/2^{n+2}}=\frac{2}{3}.\]
		Since $f_{n,k}$ is null on the rest of the points, we obtain that $\hat{f}_{n,k}(m)\leq \frac{2}{3}$ holds for every $m \in \Gamma$ with $m\neq \pm v_{n,k}$.
		
		Next, we define functionals associated to the horizontal molecules. Fix $n \in \mathbb{N}\cup\{0\}$ and $k\in\{0,\ldots,2^n-1\}$. Let us define $\varphi_{n,k} \colon [0,1]\longrightarrow \mathbb{R}$ given by
		\[ \varphi_{n,k}(x)=  \left\{
		\begin{array}{ll}
		\displaystyle\frac{3}{2^{n+2}} & \text{ if $x\in[0,\frac{k}{2^n}]$; } \vspace*{0.2cm}\\
		\displaystyle\frac{2k+3}{2^{n+2}} - \frac{x}{2} & \text{ if $x \in [\frac{k}{2^n},\frac{k+1}{2^n}]$;} \vspace*{0.2cm}\\
		\displaystyle\frac{1}{2^{n+2}} & \text{ if $x \in [\frac{k+1}{2^n},1]$. } \\
		\end{array} \right. \]
		\normalsize
		It is easy to see that $\varphi_{n,k}$ is a Lipschitz function with $\|\varphi_{n,k}\|_L = \frac{1}{2}$. Now, define $g_{n,k} \colon M_\infty \longrightarrow \mathbb{R}$ as follows
		\[ g_{n,k}((n',k'))=  \left\{
		\begin{array}{ll}
		\frac{1}{2^{n}} & \text{ if $ n'\leq n$ and  $\frac{k'}{2^{n'}}\leq \frac{k}{2^n}$; } \\
		0 & \text{ if $n'\leq n$ and $ \frac{k'}{2^{n'}}> \frac{k}{2^n}$; } \\
		\varphi_{n,k}(\frac{k'}{2^{n'}}) & \text{ if $n'>n$. } \\
		\end{array} \right. \]
		On the one hand, note that
		\[ \hat{g}_{n,k}(h_{n,k})= \frac{g_{n,k}((n,k))-g_{n,k}((n,k+1))}{\vert (n,k)-(n,k+1)\vert} = \frac{1/2^n}{1/2^n}=1.\]
		On the other hand, let us show that for every $u \in \Gamma$ with $u \neq \pm h_{(n,k)}$ we have  $$|\hat{g}_{n,k}(u)|\leq \frac{1}{2}.$$
		For this, take any vertical molecule $v_{n',k'} \in V$. Note that we have $\hat{g}_{n,k}(v_{n',k'})=0$ unless $n'=n$. On the one hand, if $k'\leq k$ we get
		\[ |\hat{g}_{n,k}(v_{n,k'})|=\frac{|g_{n,k}((n,k'))-g_{n,k}((n+1,2k'))|}{\vert(n,k')-(n+1,2k')\vert}= \frac{1/2^n - 3/2^{n+2}}{1/2^{n+1}} = \frac{1}{2}. \]
		On the other hand, if $k'\geq k+1$ we have
		\[ |\hat{g}_{n,k}(v_{n,k'})|=\frac{|g_{n,k}((n,k'))-g_{n,k}((n+1,2k'))|}{\vert (n,k')-(n+1,2k')\vert}= \frac{1/2^{n+2}}{1/2^{n+1}} = \frac{1}{2}.\]
		Finally, take any horizontal molecule $h_{n',k'}$  such that $(n',k')\neq(n,k)$. If $n'<n$, we have that
		\begin{align*} |\hat{g}_{n,k}(h_{n',k'})|&=\frac{|g_{n,k}((n',k'))- g_{n,k}((n',k'+1))|}{\vert (n',k')-(n',k'+1)\vert} \leq \frac{1/2^n}{1/2^{n-1}}=\frac{1}{2}.
		\end{align*}
		If $n'=n$, the only horizontal molecule $h$ such that $g_{n,k}(h)\neq0$ is $h=h_{n,k}$, and
		if $n'>n$ we obtain
		\[ \hat{g}_{n,k}(h_{n',k'})= \hat{\varphi}_{n,k}\left (h_{n',k'}\right ) \leq \|\varphi_{n,k}\|_L =\frac{1}{2}.\]
Actually, notice that given a pair of different points $p,q\in M_\infty\setminus \bigcup\limits_{j=0}^n A_j$ it follows that there exists a pair of different points $p',q'\in M_\infty\setminus \bigcup\limits_{j=0}^n A_j$ such that $m_{p',q'}\in\Gamma$ and $\vert \hat g_{n,k}(m_{p,q})\vert\leq \vert \hat g_{n,k}(m_{p',q'})\vert\leq \frac{1}{2}$.

		Finally, let us consider $\delta>0$ satisfying
		\[ \left (1+\frac{\varepsilon}{2}\right )(1-\delta)>1+\frac{\varepsilon}{3}.\]
		Since $\|f\|_L=\sup\left\{\|\hat{f}(m)\|\colon m \in \Gamma\right\}$, we may find $m \in \Gamma$ such that $\|\hat{f}(m)\|>1-\delta$. If $m \in H\cup V$, then consider $\hat{f}_m$ the functional associated to $m$, and if $-m \in H\cup V$, consider the same functional but multiplied by $-1$. Now, let us define
\begin{equation}\label{eq:g2}
g\colon M \longrightarrow Y, \qquad \hat{g}(x)=\hat{f}(x)+\frac{\varepsilon}{2}\hat{f}_m(x)\hat{f}(m) \quad \forall\, x \in \mathcal{F}(M).
\end{equation}
		It is clear that $\|f-g\|_L\leq \frac{\varepsilon}{2}$, so it remains to prove that $g$ strongly attains its norm. On the one hand, note that
		\[ \norm{\hat{g}(m)}=\left (1+\frac{\varepsilon}{2}\right )\norm{\hat{f}(m)}>\left (1+\frac{\varepsilon}{2}\right )(1-\delta).\]
On the other hand, if $m=m_{p_0,q_0}$ for suitable $p_0\in A_{n_{p_0}}, q_0\in A_{n_{q_0}}$ we have, by the properties of the functionals $f_{n,k}$ and $g_{n,k}$, that
$$\vert \hat f_m(m_{p,q})\vert\leq \frac{2}{3}$$
if $p$ and $q$ does not belong to $\bigcup\limits_{i=0}^{j} A_i$ for $j=\max\{n_{p_0},n_{q_0}\}$. Consequently, for $p,q\notin 	\bigcup\limits_{i=0}^{j} A_i$ we get that
\[ \norm{\hat{g}(m_{p,q})}=\norm{\hat{f}(m_{p,q})+\frac{\varepsilon}{2}\hat{f}_m(m_{p,q})\hat{f}(m)}\leq\left ( 1+\frac{\varepsilon}{3}\right)<\left (1+\frac{\varepsilon}{2}\right)(1-\delta)<\norm{\hat{g}(m)},
\]
which implies that $g$ cannot approximate its norm at arbitrarily close points, that is, it is non-local. By compactness, we deduce that $g$ strongly attains its norm.
	\end{proof}

\begin{remark}\label{remark:lluvianonlocal}
Note that, in the above proof, the map $g$ defined in both cases by, respectively, formulas \eqref{eq:g1} and \eqref{eq:g2}, is non-local.
\end{remark}

Next, we show that $\Free(\M_1)$ has property $\alpha$, giving the proof of assertion (a) of Theorem \ref{theo:lluvia}.

\begin{proposition}\label{prop:lluvia_1_alpha}
For the metric space $\M_1$ defined in Theorem \ref{theo:lluvia}, we have that $\Free(\M_1)$ has property $\alpha$.
\end{proposition}

	\begin{proof}
		Since the metric $d$ consists of summing vertical and horizontal coordinates, and $M_\infty$ is dense in $M$, it is clear that the set $\Gamma= \pm H \cup \pm V$ considered in the proof of Proposition \ref{prop:lluvia} verifies that $B_{\mathcal{F}(M)}=\cco(\Gamma)$. To see this, it is enough to note that given $(n_1,k_1), (n_2,k_2) \in M$, with $n_1<n_2$, we will have that
		\[ d((n_1,k_1), (n_2,k_2))=d((n_1,k_1), (n_2,2^{n_2-n_1}k_1)) + d((n_2,2^{n_2-n_1}k_1), (n_2,k_2)).\]
		Therefore, we need to find a set of functionals $\Gamma^*$ associated to $\Gamma$ verifying the definition of property $\alpha$. In view of the proof of Proposition \ref{prop:lluvia}, it will be enough to consider the sets
		\[ H^*=\{f_{n,k} \colon n \in \mathbb{N}\cup\{0\}, k \in \{0,\ldots, 2^n-1\}\},\]
		\[V^*=\{g_{n,k} \colon n \in \mathbb{N}\cup\{0\}, k \in \{0,\ldots, 2^n\}\},\]
		and $\Gamma^*=\pm H^* \cup \pm V^*$, to obtain that the pair $(\Gamma, \Gamma^*)\subseteq \mathcal{F}(M) \times \Lip(M,\mathbb{R})$ satisfies the statements of property $\alpha$ with constant $\frac{2}{3}$.
\end{proof}
	
The last part of the proof of Theorem \ref{theo:lluvia} is contained in the next proposition.

\begin{proposition}\label{prop:nocufe}
Let $\M_2$ be the metric space given in Theorem \ref{theo:lluvia}. If $\Gamma\subseteq \Mol{\M_2}$ is a subset satisfying that $\overline{\co}(\Gamma)=B_{\Free(\M_2)}$, then $\Gamma$ is not a uniformly strongly exposed set.
\end{proposition}

\begin{proof}
Pick such a subset $\Gamma$. By the paragraph below Corollary~3.10 in \cite{ccgmr}, it follows that $\dent{B_{\Free(\M_2)}}\subseteq \overline{\Gamma}$. Now, for every $n\in\mathbb N$, consider the points
    $$x_n:=\left(0,\frac{1}{2^{n}}\right) \quad \text{and} \quad  y_n:=\left(1,\frac{1}{2^{n+1}}\right).$$
It is clear that the pair $(x_n,y_n)$ fails property (Z), so $m_{x_n,y_n}$ is a strongly exposed point. Furthermore, \cite[Lemma 1.3]{ccgmr} implies that $m_{x_n,y_n}$ is an isolated point in $\Mol{\M_2}$, so $m_{x_n,y_n}\in\Gamma$. We will prove that the set $\{m_{x_n,y_n}\colon n\in\mathbb N\}$ is not uniformly strongly exposed. To do so, we will use the criterium given in \cite[Proposition 3.6]{ccgmr}. Let $z_n = (\frac{1}{2}, \frac{1}{2^{n+1}})$. Note that
\[ \min\{d(x_n,z_n), d(z_n,y_n)\}=\frac12\]
and
\begin{align*}
2(x_n,y_n)_{z_n} = \left(\frac{1}{4}+\frac{1}{2^{2n+2}}\right)^{1/2}+\frac{1}{2}-\left(1+\frac{1}{2^{2n+2}}\right)^{1/2}\longrightarrow 0
\end{align*}
as $n\to\infty$. Now, \cite[Proposition 3.6]{ccgmr} finishes the proof.
\end{proof}

	\begin{remark}
{\slshape		Note that the Lipschitz-free space over the metric spaces $\M_p$ in Theorem \ref{theo:lluvia} fails the RNP for $p=1,2$ since $\M_p$ contains an isometric copy of $[0,1]$.}\  Even more, there is a $1$-Lipschitz retraction $r\colon \M_p\longrightarrow [0,1]$, and this implies that $\mathcal F(\M_p)$ even contains a complemented copy of $L_1[0,1]$.
	\end{remark}

Let us point out some consequences from the previous examples.

\begin{remark}
In \cite[p.~115]{Godefroy-survey-2015} the author asks whether, given a compact metric space $M$, the density of $\SA(M,\mathbb R)$ in $\Lip(M,\mathbb R)$ implies that $\lip(M)$ strongly separates the points of $M$ (see \cite[p.~110]{Godefroy-survey-2015} for details). Note that the spaces $\M_p$ provide a counterexample for $p=1,2$, because if $\lip(M)$ strongly separates the points of $M$ then, in particular, $\mathcal F(M)$ is isometric to a dual Banach space. However, $\mathcal F(\M_p)$ is not even isomorphic to any dual Banach space as it is a separable Banach space failing the RNP.
\end{remark}

Let $M$ be a metric space. In \cite[Section 3.4]{ccgmr} it is stated to be unknown whether the density of $\SA(M,Y)$ in $\Lip(M,Y)$, for every Banach space $Y$, implies any of the following properties:
\begin{enumerate}
\item $\mathcal F(M)$ has the RNP.
\item $B_{\mathcal F(M)}=\overline{\co}(S)$, where $S$ is a uniformly strongly exposed set.
\item $\mathcal F(M)$ has property quasi-$\alpha$.
\end{enumerate}
From our results it follows that the density of $\SA(M,Y)$ for every $Y$ does not imply any of the above properties. Indeed, an example failing simultaneously all the above properties can be given, as the following example shows.

\begin{example}\label{example:frankenstein}
{\slshape There is a complete metric space $M$ satisfying that $\SA(M,Y)$ is dense in $\Lip(M,Y)$ for every Banach space $Y$ and such that $\F{M}$ fails the RNP, property $\alpha$, property quasi-$\alpha$, and it does not contain any norming uniformly strongly exposed set.}
\end{example}

We need the following easy result which we prove since we have not been able to find a reference.

\begin{lemma}\label{lemma:qalpha}
Let $X$, $Y$ be two Banach spaces and write $Z:=X\oplus_1 Y$.
\begin{enumerate}
\item[(a)] If $Z$ has property quasi-$\alpha$, then $X$ has property quasi-$\alpha$.
\item[(b)] Assume that $S_{Z}$ contains a uniformly strongly exposed subset $\Gamma$ such that $\overline{\co}(\Gamma)=B_{Z}$. Then, $S_X$ contains a uniformly strongly exposed subset $\Delta$ such that $\overline{\co}(\Delta)=B_X$.
\end{enumerate}
\end{lemma}

\begin{proof}
(a). Let $A_Z:=\bigl\{(x_\lambda,y_\lambda)\colon \lambda\in \Lambda_Z\bigr\}$, $\bigl\{(x_\lambda^*,y_\lambda^*)\colon \lambda\in \Lambda_Z\bigr\}$ and $\rho_Z:\Lambda_Z\longrightarrow \R$ be the sets and the function given by the definition of property quasi-$\alpha$. Since $$A_Z\subseteq \ext{B_Z}=\left(\ext{B_X}\times \{0\} \right)\cup \left(\{0\}\times \ext{B_Y} \right),$$ we may consider
$$A_X:=A_Z\cap (B_X\times \{0\})\equiv \{x_\lambda\colon \lambda\in \Lambda_X\}$$
for convenient non-empty subset $\Lambda_X$ of $\Lambda_Z$. Let us see that $X$ has property quasi-$\alpha$ witnessed by the sets $A_X$ and $\{x_\lambda^*\colon \lambda\in \Lambda_X\}$ and the function $\rho_X:=\rho_Z|_{\Lambda_X}\colon \Lambda_X\longrightarrow \R$. Indeed:
\begin{itemize}
    \item For every $\lambda\in \Lambda_X$, we have that
    \[       x_\lambda^*(x_\lambda)=(x_\lambda^*,y_\lambda^*)(x_\lambda, 0)=1.
  \]
\item For $\mu\neq \lambda$, we have that
     \[
              \vert x_\lambda^*(x_\mu)\vert=\vert (x_\lambda^*,y_\lambda^*)(x_\mu,0)\vert\leq \rho_Z(\lambda)=\rho_X(\lambda)<1.
         \]
     \item Given $e^{**}\in \ext{B_{X^{**}}}$, then $(e^{**},0)\in \ext{B_{Z^{**}}}$, so we can find $A_{(e^{**},0)}\subset A_Z$ and $\omega\in \{-1,1\}$ such that $$\omega (e^{**},0)\in \overline{J_Z(A_{(e^{**},0)})}^{\,w^*}$$
     and $\sup\{\rho_Z(\lambda)\colon  (x_\lambda,y_\lambda)\in A_{(e^{**},0)} \}<1$; we define $A_{e^{**}}=\pi(A_{(e^{**},0)})$ (where $\pi\colon Z\longrightarrow X$ denotes the natural projection) and observe that
     \begin{align*}
         \omega e^{**}=\omega\pi^{**}(e^{**},0)&\in \pi^{**}\Bigl(\overline{J_Z(\Lambda_{(e^{**},0)})}^{\,w^*}\Bigr) \mathop{\subseteq}\limits^{\diamondsuit} \ \overline{[\pi^{**}\circ J_Z](\Lambda_{(e^{**},0)})}^{\,w^*} \\ &  =\overline{[J_X\circ\pi](\Lambda_{(e^{**},0)})}^{w^*} =\overline{J_X(\Lambda_{e^{**}})}^{\,w^*},
     \end{align*}
     where the inclusion $\diamondsuit$ follows from the weak-star continuity of $\pi^{**}$.
     Now, it is clear that $$\sup\{\rho_X(\lambda)\colon  x_\lambda\in A_{e^{**}}\} \leq \sup\{\rho_Z(\lambda)\colon  (x_\lambda,y_\lambda)\in A_{(e^{**},0)} \}<1.$$
\end{itemize}

(b). Since $\Gamma$ is made of strongly exposed points of $B_{Z}$, then every element $(x,y)\in \Gamma$ satisfies that either $\Vert x\Vert=1$ and $y=0$ or $x=0$ and $\Vert y\Vert=1$. Define
$$\Delta:=\{x\in S_X\colon (x,0)\in \Gamma\}.$$
Given $(x,0)\in \Gamma$, the definition of uniformly strongly exposing set yields a strongly exposing functional $(f_x,g_x)\in S_{Z^*}$ associated to $(x,0)$. Notice that $\Vert f_x\Vert=1$ since $1=\langle(f_x,g_x),(x,0)\rangle=f_x(x)$. It is clear that $\Delta$ is a uniformly strongly exposed set by making use of the fact that it is identified with a subset of $\Gamma$ which is a uniformly strongly exposed set. The fact that $\overline{\co}(\Delta)=B_X$ follows from the fact that $\overline{\co}(\Gamma)=B_Z$ and the shape of the unit ball of an $\ell_1$-sum.
\end{proof}

\begin{proof}[Proof of Example \ref{example:frankenstein}]
Let us consider, following the notation of \cite{wea5}, the metric space
$$M:=\M_2\coprod [0,1]^\frac{1}{2}.$$
It is known (c.f.\ e.g.\ \cite[Proposition 1.4]{ccgmr} or \cite[Proposition 3.9]{wea5}) that $\mathcal F(M)=\mathcal F(\M_2)\oplus_1\mathcal F([0,1]^\frac{1}{2})$. By Proposition \ref{prop:lluvia}, (P2.2) on page 3, and \cite[Proposition 5.6]{cm-preprint}, we get that $\SA(M,Y)$ is dense in $\Lip(M,Y)$ for every Banach space $Y$. Also $\mathcal F(M)$ fails property quasi-$\alpha$ because $\mathcal F([0,1]^\frac{1}{2})$ fails property quasi-$\alpha$ \cite[Example 3.22]{ccgmr} and we may use Lemma~\ref{lemma:qalpha}.a. Further, $\mathcal F(M)$ fails the RNP because it contains an isometric copy of $L_1[0,1]$. Finally, there is no uniformly strongly exposed set $\Gamma\subseteq S_X$ such that $\overline{\co}(\Gamma)=B_{\mathcal F(M)}$ by Proposition \ref{prop:nocufe} and Lemma \ref{lemma:qalpha}.b.
\end{proof}

Although we know that the density of $\SA(M,Y)$ in $\Lip(M,Y)$ for every Banach space $Y$ does not imply the RNP by the examples given in Theorem \ref{theo:lluvia}, we would like to take a closer look to the possible relationship between these two properties. Since the RNP is an isomorphic property, if a Banach space $X$ has the RNP, then $X$ has Lindenstrauss property A for all equivalent norms. A classical result in theory of norm attaining operators is that the converse also holds, i.e.\ a Banach space $X$ has the RNP if, and only if, every equivalent renorming of $X$ has Lindenstrauss property A (a result by Bourgain and Huff, see \cite{Huff}). On the other hand, the RNP of $X$ is also equivalent to the dentability of the unit ball of each equivalent renorming of $X$ and so, equivalent to the dentability of the unit balls of all equivalent renormings of closed subspaces of $X$. It is then natural to wonder whether there is a Lipschitz version for strongly norm attaining Lipschitz maps. Namely, we consider the following assertions on a complete metric space $M$:
\begin{enumerate}
    \item[(1)] $\mathcal F(M)$ has the RNP.
    \item[(2)] $\overline{\SA(N',Y)}=\Lip(N',Y)$ for every metric space $N'$ bi-Lipschitz equivalent to a closed subset of $M$ and every Banach space $Y$.
    \item[(3)] For every metric space $N'$ bi-Lipschitz equivalent to a closed subset of $M$, the unit ball of $\Free(N')$ is dentable.
    \end{enumerate}
Then, we have that (1)$\Rightarrow$(2)$\Rightarrow$(3). Indeed, note that (1)$\Rightarrow$(2) follows from \cite[Proposition 7.4]{lppr} since, in this case, $\Free(N')$ has the RNP. In order to prove that (2)$\Rightarrow$(3), notice that given a metric space $M$ under the assumption (2), it follows that, if $N'$ is a metric space bi-Lipschitz equivalent to a subspace of $M$, then $\Free(N')$ has Lindenstrauss property A, so $B_{\Free(N')}$ is dentable (see \cite[Proposition 1]{bou77}). We do not know whether the implication (2)$\Rightarrow$(1) holds. On the other hand, another natural question is whether the implication (3)$\Rightarrow$(1) holds. However, this is not longer true, and even the implication (3)$\Rightarrow$(2) fails.

\begin{example}
Let $M$ be a nowhere dense closed subset of $[0,1]$ whose Lebesgue measure is positive. Then, given any metric space $N'$ bi-Lipschitz equivalent to a closed subspace of $M$, it follows that $N'$ is not geodesic (because it is disconnected). Consequently, $B_{\Free(N')}$ has strongly exposed points by \cite[Corollary~5.11]{gprdauga} and so, it is dentable. However, $\SA(M,\mathbb{R})$ is not dense in $\Lip(M,\mathbb{R})$ by \cite[Theorem~2.3]{ccgmr}.
\end{example}

Let us finally mention that we do not know whether any of the following two properties, which are weaker than (2) above, implies $\Free(M)$ to have the RNP:
\begin{enumerate}
    \item[(4)] $\overline{\SA(N,Y)}=\Lip(N,Y)$ for every closed subset $N$ of $M$ and every Banach space $Y$.
    \item[(5)] $\overline{\SA(M',Y)}=\Lip(M',Y)$ for every metric space $M'$ bi-Lipschitz equivalent to $M$ and every Banach space $Y$.
\end{enumerate}

\section{Consequences of the density of strongly norm attaining Lipschitz maps}\label{section:abiden}

It is known that the density of the set of strongly norm attaining Lipschitz maps from a metric space $M$ to a Banach space $Y$ is stronger than the density of $\NA(\F{M},Y)$ as, for instance, $\NA(\F{M},\R)$ is always dense by the Bishop-Phelps theorem, but there are many metric spaces $M$ for which $\SA(M,\R)$ is not dense. Our aim in this section is to deepen in this line, showing that the density of $\SA(M,Y)$ has important consequences. In particular, we will show that some results of Lindenstrauss and Bourgain can be somehow improved in the setting of Lipschitz-free spaces.

The next two technical results are the key to getting all the goals of this section.

\begin{lemma}\label{Lemma:dichext} Let $M$ be a complete metric space, let $f\in \SA(M, \mathbb R)$, and let $m_{p,q}\in \Mol M $ such that $\hat{f}(m_{p,q})=\|f\|_L$. Consider the set
		\[ F_{p,q}:=\{(x,y) \in M^2 \colon x\neq y,\ d(p,q)=d(p,x)+d(x,y)+d(y,q)\}. \]Then, either there is $(x,y)\in F_{p,q}$ such that $m_{x,y}\in \ext{B_{\mathcal F(M)}}$ or there is an isometric embedding $\phi \colon [0,d(p,q)]\longrightarrow M$ for which $\phi(0)=p$ and $\phi(d(p,q))=q$.
	\end{lemma}	
	
	\begin{proof}
		First, note that $(p,q) \in F_{p,q}$ and so $F_{p,q}$ is not empty. Assume that $m_{x,y}$ is not an extreme point whenever $(x,y)\in F_{p,q}$. By \cite[Theorem 1.1]{ap}, for every $(x,y)\in F_{p,q}$ there is $z\in M$ such that $d(x,z)+d(z,y)=d(x,y)$. 
		
		The rest of the proof is just a small modification of the one of Proposition 4.1 in \cite{gprdauga}. Our aim is to show that there is an isometry $\phi\colon [0, d(p,q)] \longrightarrow M$ such that $\phi(0) =p$ and $\phi(d(p,q)) =q$. Consider the set $\mathcal A$ of all $(A, \psi)$, where $\{0,d(p,q)\}\subset A \subset [0, d(p,q)]$ is closed and $\psi \colon A \longrightarrow M$ is an isometry such that $\psi (0) =p$, $\psi(d(p,q)) =q$, and $(\psi(t),\psi(s))\in F_{p,q}$ for every $t$, $s\in A$ with $t<s$. Consider the following partial order ``$\leq$'' on $\mathcal A$: $(A, \psi) \leq (B, \xi)$ if $A \subset B$ and $\xi\restricted_A=\psi$. Clearly $\mathcal A\neq \emptyset$.
		
		We claim that every chain in $\mathcal A$ has an upper bound. Indeed, let $(A_i, \psi_i)_{i\in I}$ be a chain in $\mathcal A$. Take $A=\overline{\bigcup_{i\in I}A_i}$ and $\psi(x):=\psi_i(x)$ if $x\in A_i$. By completeness, we can extend $\psi$ uniquely to an isometry defined on $A$. Moreover, let $t,s\in A$, $t< s$. Then there are sequences $\{t_n\}, \{s_n\}$ in $\bigcup_{i\in I} A_i$ such that $t_n< s_n$, $t_n\to t$ and $s_n\to s$. Then
		\[\bigl|\hat{f}(m_{\psi(t), \psi(s)})\bigr| = \lim_n \bigl|\hat{f}( m_{\psi(t_n), \psi(s_n)})\bigr| = \Vert f\Vert_L\]
		since $(\psi(t_n), \psi(s_n))\in F_{p,q}$ for every $n$. Thus $(\psi(t), \psi(s))\in F_{p,q}$. This means that  $(A,\psi)\in \mathcal A$.
		
		Now, let $(A,\phi)$ be a maximal element in $\mathcal A$. Assume that there are $a,b\in A$, $a<b$ such that $(a,b)\cap A=\emptyset$. Since $(\phi(a), \phi(b))\in F_{p,q}$, we have that $m_{\phi(a),\phi(b)}$ is not an extreme point. Consequently, there is $z\in [\phi(a),\phi(b)]\setminus\{\phi(a), \phi(b)\}$. Then, we extend $\phi$ defining $\phi(a +d(\phi(a), z)) :=z$. Let us show that this map contradicts the maximality of $(A, \phi)$. It is clear that $\phi$ is still an isometry with $\phi(0)=p$ and $\phi(d(p,q))=q$. It remains to prove that $(\phi(t),\phi(s)) \in F_{p,q}$ for every $t \in A$ with $t<s$. Clearly, we may assume that either $\phi(s)=z$ or $\phi(t)=z$. Let's assume the first case holds, since the other one is similar. Since $t\in A$ and $t\leq \phi^{-1}(z)$, we have $t\leq a$. Then, we have that $\phi(a)\in[\phi(t),z]$  
		and it is clear that $\phi(b)\in[z,q]$.  
		Joining these two equalities we obtain
		\begin{multline*}
		 d(p,\phi(t))+d(\phi(t),z)+d(z,q) \\
		=d(p,\phi(t))+d(\phi(t),\phi(a))+d(\phi(a),z)+d(z,\phi(b))+d(\phi(b),q).
		\end{multline*}
		Recall that $z\in [\phi(a),\phi(b)]$ 
		and
		$\phi(a)\in[\phi(t), \phi(b)]$, so we have that
		\[
		d(p,\phi(t))+d(\phi(t),z)+d(z,q)=d(p,\phi(t))+d(\phi(t),\phi(b))+d(\phi(b),q)=d(p,q),\]
		since $(\phi(t),\phi(b)) \in F_{p,q}$. This means that $(\phi(t), z)\in F_{p,q}$.
	\end{proof}

\begin{lemma} \label{lemma:gamma} Let $M$ be a complete metric space. Let $\Gamma$ be a balanced subset of $S_{\mathcal F(M)}$ and denote $N_\Gamma(M)=\left\{f\in \Lip(M,\mathbb R)\colon \sup_{m\in\Gamma} |\hat{f}(m)|=\norm{f}_L\right\}$. Suppose that the set
	\[\bigl\{f\in \Lip(M,\mathbb R)\colon \hat{f}(m_{x,y})=\norm{f}_L \text{ for some } m_{x,y}\in \Mol{M}\cap \ext{B_{\mathcal F(M)}}\bigr\}\]
is contained in $N_\Gamma(M)$ and that $\SA(M,\mathbb R)$ is dense in $\Lip(M,\mathbb{R})$. Then, $$\SA(M,\mathbb R)\subset N_{\Gamma}(M)$$ and so, $B_{\mathcal F(M)} = \overline{\co}(\Gamma)$.	
\end{lemma}

\begin{proof} Assume that $h_1\in \SA(M,\mathbb R)\setminus N_\Gamma(M)$, with $\norm{h_1}_L=1$. Take $0<\delta<1$ in such a way that $\sup_{m\in\Gamma} |\hat{h}_1(m)|=1-\delta$. Now, if $h_1$ strongly attains its norm at a molecule $m_{p,q}$, by applying Lemma~\ref{Lemma:dichext} and taking into account that $h_1$ does not attain its norm at any extreme molecule, we find an isometry $\phi \colon [0,d(p,q)] \longrightarrow M$ satisfying $\phi(0)=p$ and $\phi(d(p,q))=q$. Consider $u_0 \colon \phi([0,d(p,q)]) \longrightarrow [0,d(p,q)]$ its inverse map and an extension $u \colon M\longrightarrow [0,d(p,q)]$ of $u_0$ such that $\|u\|_L=1$. Note that such extension exists thanks to McShane extension theorem. On the other hand, let $C\subseteq [0,d(p,q)]$ be a fat Cantor set, that is, a measurable closed subset with $\lambda(C)>(1-\delta)d(p,q)$ such that for each nontrivial interval $I\subseteq[0,d(p,q)]$ there exists a nontrivial interval $J\subseteq I$ such that $J\cap C =\emptyset$. Let us consider $\varphi\colon[0,d(p,q)]\longrightarrow \mathbb{R}$ given by
	\[ \varphi(t)=-\int_0^t \chi_C(s)\,ds \quad \forall \, t \in [0,d(p,q)].\]
	We define $h_2 \colon M \longrightarrow\mathbb{R}$ by $h_2=\varphi \circ u$ and $f \colon M \longrightarrow \mathbb{R}$ by $f=\frac{1}{2}(h_1+h_2)$. It is clear that $$\|f\|_L\leq \frac{1}2(\|h_1\|_L+\|h_2\|_L) =1.$$ Moreover, \[ \|f\|_L\geq \hat{f}(m_{p,q})=\frac{1}{2}(\hat{h}_1(m_{p,q})+\hat{h}_2(m_{p,q}))= \frac{1}{2}\left (1+\frac{\lambda(C)}{d(p,q)}\right )>1-\frac{\delta}{2}.\]
	
	On the other hand,
	\[\sup_\Gamma \hat{f} \leq \frac{1}{2}(\sup_\Gamma \hat{h}_1+\sup_\Gamma \hat{h}_2) = \frac{1}{2}(2-\delta)=1-\frac{\delta}{2}.\]
	Therefore, $f\notin N_\Gamma(M)$. Take $\varepsilon>0$ such that $\norm{f}_L-\varepsilon>1-\frac{\delta}{2}$. Since $N_\Gamma(M)$ is closed and $\SA(M,\mathbb R)$ is dense there is $g\in \SA(M,\mathbb R)\setminus N_\Gamma(M)$ such that  $\|g\|_L=\|f\|_L$ and $\|f-g\|_L<\varepsilon$. Consider $m_{x,y} \in \Mol{M}$ for which $\hat{g}(m_{x,y})=\|g\|_L$. Since $g\notin N_\Gamma(M)$, we have that $\hat{g}$ does not attain its norm at any extreme molecule of $B_{\mathcal F(M)}$. In particular, by applying Lemma~\ref{Lemma:dichext} we obtain that there exists an isometry $\phi'\colon [0,d(x,y)] \longrightarrow M$ satisfying $\phi'(0)=x$ and $\phi'(d(x,y))=y$. Notice that
	\[\hat{h}_2(m_{x,y})\geq 2\hat{f}(m_{x,y})-1\geq 2(\hat{g}(m_{x,y})-\varepsilon)-1\geq 2(\norm{f}_L-\varepsilon)-1>1-\delta,\]
	 from where $u(x)\neq u(y)$. Hence, we can find different points $a$, $b$ of $\phi'([0,d(x,y)])\subseteq M$ such that $d(x,y)=d(x,a)+d(a,b)+d(b,y)$ and $[u(a),u(b)]\cap C=\emptyset$. Since $g$ attains its norm at $m_{x,y}$ it follows that $\hat{g}(m_{a,b})=\|g\|_L$ and so, $\hat{f}(m_{a,b})>\|f\|_L-\varepsilon$. As before, this implies that $\hat{h}_2(m_{a,b})>1-\delta$, whereas the fact that $[u(a),u(b)]\cap C=\emptyset$ implies that $\hat{h}_2(m_{a,b})=0$, leading to a contradiction.
	
	 This shows that $\SA(M,\mathbb R)\subset N_{\Gamma}(M)$. Now, Hahn-Banach theorem yields that $B_{\mathcal F(M)}=\overline{\co}(\Gamma)$.
\end{proof}

The first main result of the section is the following.

\begin{theorem}
\label{theorem:closedconvexhullextremepoints}
Let $M$ be a complete metric space. If $\SA(M,\R)$ is dense in $\Lip(M,\R)$, then
$$
B_{\F{M}}=\overline{\co}\Bigl(\ext{B_{\mathcal F(M)}}\cap \Mol{M}\Bigr).
$$
\end{theorem}

\begin{proof} Apply Lemma \ref{lemma:gamma} with $\Gamma= \Mol{M}\cap\ext{B_{\mathcal F(M)}}$.
\end{proof}

We would like to observe that Theorem \ref{theorem:closedconvexhullextremepoints} somehow improves, in the case of Lipschitz-free spaces, a result by Lindenstrauss \cite{linds63}.

\begin{remark}\label{remark:Lidenstrauss-remark}
{\slshape Let $M$ be a complete metric space. If $\SA(M,Y)$ is dense in $\mathcal{L}(\F{M},Y)$ for \emph{some} Banach space $Y$, then}
$$
B_{\F{M}}=\overline{\co}\bigl(\ext{B_{\mathcal F(M)}}\bigr).
$$
Indeed, this follows from Theorem \ref{theorem:closedconvexhullextremepoints}, as the density of $\SA(M,Y)$ in $\Lip(M,Y)$ for some $Y$ implies the density of $\SA(M,\R)$ in $\Lip(M,\R)$ by \cite[Proposition~4.2]{cm}.

Compare this result with the following one by Lindenstrauss \cite[Theorem 2.i]{linds63}: if $X$ is a Banach space which admits a strictly convex renorming (for instance, if $X$ is separable) such that $\NA(X,Y)$ is dense in $\mathcal{L}(X,Y)$ for \emph{all} Banach spaces $Y$, then $B_X=\overline{\co}(\ext{B_{X}})$.
\end{remark}

The next result deals with strongly norm attaining vector-valued Lipschitz maps. In the case of real-valued maps, it improves Theorem \ref{theorem:closedconvexhullextremepoints} for metric spaces not containing isometric copies of the unit interval.

\begin{proposition}\label{prop:centraextremos}
Let $M$ be a complete metric space which does not contain any isometric copy of $[0,1]$ and let $Y$ be a Banach space. Then, $\SA(M,Y)$ coincides with the set
$$
\left\{f\in \Lip(M,Y)\colon \Vert \hat{f}(m_{x,y})\Vert =\|f\|_L \text{ for some } m_{x,y}\in \Mol{M}\cap \ext{B_{\mathcal F(M)}}\right\}.
$$
In particular, if $\SA(M,Y)$ is dense in $\Lip(M,Y)$, then so is the set
$$
\left\{T\in \mathcal{L}(\F{M},Y)\colon T \text{ attains its norm at some element of } \ext{B_{\mathcal F(M)}} \right\}
$$
in $\mathcal{L}(\F{M},Y)$.
\end{proposition}

\begin{proof}
Pick $f\in \SA(M,Y)$. Hence there exists $u,v\in M, u\neq v$ and $y^*\in S_{Y^*}$ such that $[y^*\circ \hat{f}](m_{u,v})=\Vert f\Vert_L$. Since $M$ does not contain any isometric copy of $[0,1]$, then Lemma \ref{Lemma:dichext} applies to get a molecule $m_{x,y}\in \ext{B_{\mathcal F(M)}}$ such that $\bigl[y^*\circ \hat{f}\,\bigr](m_{x,y})=\Vert f\Vert_L$. From here it is clear that $f$ strongly attains its norm at the pair $(x,y)$, and we are done.
\end{proof}

\begin{remark}
The assumption that $\SA(M,\mathbb R)$ is dense cannot be removed in Theorem \ref{theorem:closedconvexhullextremepoints}: {\slshape  let $M$ be a fat Cantor set in $[0,1]$, which clearly does not contain any isometric copy of $[0,1]$; then, $B_{\mathcal F(M)}\neq \overline{\co}\bigl(\ext{B_{\mathcal F(M)}}\bigr)$.}\  Indeed, it is known that $\mathcal F(M)\equiv L_1[0,1]\oplus_1 \ell_1$ \cite[pp.~4315]{godard} so, under that identification, any $(f,g)\in \ext{B_{\mathcal F(M)}}$ is actually of the form $(0,g)$, for certain $g\in \ext{B_{\ell_1}}$, and hence $\overline{\co}\bigl( \ext{B_{\mathcal F(M)}}\bigr) \subseteq \{0\}\oplus_1 \ell_1$. This reproves the fact that $\SA(M,\mathbb R)$ is not dense in $\Lip(M,\mathbb R)$ which follows from  \cite[Theorem 2.3]{ccgmr}.
\end{remark}

In the sequel we will obtain improvements of Theorem~\ref{theorem:closedconvexhullextremepoints} and Proposition \ref{prop:centraextremos} in the case of compact metric spaces obtaning, as a consequence, an affirmative answer to (Q3) in the case of compact metric spaces. To begin with, let us exhibit the second main result of the section.

\begin{theorem}\label{theo:centraestruct}
Let $M$ be a compact metric space which does not contain any isometric copy of $[0,1]$ and let $Y$ be a Banach space. Then, the following assertions are equivalent:
\begin{enumerate}
\item[(i)] $\SA(M,Y)$ is dense in $\Lip(M,Y)$.
\item[(ii)] The set of absolutely strongly exposing operators from $\mathcal F(M)$ to $Y$ is dense in $\mathcal L(\mathcal F(M),Y)$.
\item[(iii)] The set of non-local $Y$-valued Lipschitz maps is dense in $\Lip(M, Y)$.
\end{enumerate}
\end{theorem}

Before proving the result, let us present its main consequence, which follows immediately from the fact that the set of non-local $Y$-valued Lipschitz maps is an open set (indeed, if $f\in \Lip(M,Y)$ is a non-local Lipschitz map, we can find $\varepsilon>0$ such that $$\sup\limits_{0<d(x,y)<\varepsilon}\Vert \hat{f}(m_{x,y})\Vert<\|f\|_L-\varepsilon;$$ then, the whole $B(f,\frac{\varepsilon}{3})$ is made of non-local Lipschitz maps).

\begin{corollary}\label{coro:SNA-contains-open}
Let $M$ be a compact metric space which does not contain any isometric copy of $[0,1]$ and let $Y$ be a Banach space. If $\SA(M,Y)$ is dense in $\Lip(M,Y)$, then $\SA(M,Y)$ (and, in particular, $\NA(\F{M},Y)$) contains an open dense subset.
\end{corollary}

In the case when $\F{M}\equiv \ell_1$ or, more generally, when $\F{M}$ has property $\alpha$ witnessed by a set $\Gamma\subset S_{\mathcal F(M)}$, it is easy to see the result from the proof of \cite[Proposition 1.3.a]{schachermayer}: indeed, it is proved there that the set of those operators $T\colon \mathcal F(M)\longrightarrow Y$ such that $\sup_{y\in \Gamma\setminus\{\pm x\}} \norm{Ty}<\norm{T}=\norm{Tx}$ for some $x\in \Gamma$ is dense and, on the other hand, it is clearly open as $\overline{\co}(\Gamma)=B_{\mathcal{F}(M)}$.

A specially interesting particular case of Corollary \ref{coro:SNA-contains-open} is the one in which $\F{M}$ has the RNP. In this case, $M$ does not contain copies of $[0,1]$ (otherwise, $L_1[0,1]$ would be a subspace of $\F{M}$) and $\SA(M,Y)$ is dense in $\Lip(M,Y)$ by \cite[Proposition 7.4]{lppr}.

\begin{corollary}\label{coro:RNP-SNA-open}
Let $M$ be a compact metric space for which $\F{M}$ has the RNP. Then, for every Banach space $Y$, $\SA(M,Y)$ (and so $\NA(\F{M},Y)$) contains a dense open subset.
\end{corollary}

Compare the result above with the one by Bourgain \cite[Theorem~5]{bou77}: if $X$ is a Banach space with the RNP, then for every Banach space $Y$, $\NA(X,Y)$ contains a dense $G_\delta$ subset of $\mathcal{L}(X,Y)$. Actually, by the cited results of Bourgain \cite{bou77}, $\NA(X,\R)$ contains a dense $G_\delta$ subset of $X^*$ whenever $X$ has the RNP. Moreover, in this case, $X^*=\NA(X,\R)-\NA(X,\R)$ (see the proof of \cite[Proposition 2.23]{Band-Godefroy}, for instance). But this is far from implying that $\NA(X,\R)$ contains an open set. Let us comment that the result in Corollary \ref{coro:RNP-SNA-open} is somehow unexpected, even for functionals, as the following remark shows.

\begin{remark}
\label{remark:abidensoraro} {\slshape The presence of open subsets in the set of norm attaining operators or even functionals is a rare phenomenon.}
\begin{enumerate}
\item[(a)] If $X$ is a non-reflexive Banach space, then there always exists an equivalent renorming $\tilde{X}$ of $X$ such that $\NA(\tilde{X},\R)$ has empty interior \cite{AcoKadet}. Therefore, the RNP is not enough in general to get that the set of norm attaining operators (or even functionals) has non empty interior.
\item[(b)] Even for the Lipschitz-free norm, the hypothesis of density of strongly norm attaining Lipschitz functions is important to get that the set of norm attaining functionals has non-empty interior, as the following example shows: {\slshape For $M=[0,1]$, the norm interior of the set $\NA(\Free(M),\R)$ is empty.}\ Indeed, recall that $\Free([0,1])\equiv L_1[0,1]$ (c.f. e.g. \cite[Example 2.1]{Godefroy-survey-2015}). Now, the result follows immediately since $L_1[0,1]^*\setminus \NA(L_1[0,1],\mathbb R)$ is dense in $L_1[0,1]^*$ by \cite[Theorem~2.7]{AcoAizAronGar}.
\end{enumerate}
\end{remark}

On the other hand, we do not know whether the hypothesis that $M$ does not contain isometric copies of $[0,1]$ can be dropped in Theorem \ref{theo:centraestruct}. The only metric spaces $M$ which we know that contain $[0,1]$ and for which $\SA(M,\R)$ is dense in $\Lip(M,\R)$ are the ones given in Theorem \ref{theo:lluvia}. As a matter of facts, the three assertions of Theorem \ref{theo:centraestruct} and the thesis of Corollary \ref{coro:RNP-SNA-open} hold for them.

\begin{example}
{\slshape Let $\M_p$ be the metric spaces defined in Theorem \ref{theo:lluvia} which contain $[0,1]$ and satisfy that $\SA(\M_p,Y)$ is dense in $\Lip(\M_p,Y)$ for every Banach space $Y$ for $p=1,2$. Then, for every Banach space $Y$, the set of non-local $Y$-valued Lipschitz maps is dense in $\Lip(\M_p,Y)$ for $p=1,2$ by Remark \ref{remark:lluvianonlocal}. In particular, $\SA(\M_p,Y)$ contains a dense open set for $p=1,2$.}
\end{example}

Let now prove Theorem \ref{theo:centraestruct}. We need a number of preliminary results which could be of independent interest. First, we prove the abundance of non-local Lipschitz maps when the set of strongly norm attaining maps is dense, in the compact setting.

\begin{lemma}\label{lemma:approxnoloc}
Let $M$ be a compact metric space, let $Y$ be a Banach space, and $f\in S_{\Lip(M,Y)}$. Assume that there exists $m_{x,y}\in \ext{B_{\mathcal F(M)}}$ such that $\Vert \hat{f}(m_{x,y})\Vert =1$. Then, for every $\varepsilon>0$, there exists a non-local Lipschitz map $\phi:M\longrightarrow Y$ such that $\Vert f-\phi\Vert_L<\varepsilon$.
\end{lemma}

\begin{proof} Since $\Vert \hat f(m_{x,y})\Vert=1$ then we can find $y^*\in S_{Y^*}$ such that $[y^*\circ \hat{f}](m_{x,y})=1$. By assumption, $m_{x,y}$ is an extreme point. Hence, by \cite[Theorem 4.2]{ag} it is a preserved extreme point or, equivalently by \cite[Theorem 2.4]{lppr}, $m_{x,y}$ is a denting point. Fix $0<\delta<\frac12$ and find a slice $S=S(B_{\mathcal F(M)},\hat{h},\beta)$ with $h\in S_{\Lip(M,\R)}$ and $\beta>0$, containing $m_{x,y}$ and such that $\diam(S)<\delta$. Select $z\in S_Y$ such that $y^*(z)\hat{h}(m_{x,y})>1-\beta$ and define
$$
\hat{\phi}:=\hat{f}+\varepsilon \hat{h}\otimes z,
$$
where $\hat h\otimes z(m_{u,v}):=\hat h(m_{u,v})z$ for every $u,v\in M, u\neq v$. It is clear that $\Vert f-\phi\Vert_L<\varepsilon$. Let us now prove that $\phi$ is not local. To begin with, notice that
$$
\Vert \phi\Vert_L\geq [y^*\circ \hat{f}](m_{x,y})+\varepsilon \hat{h}(m_{x,y})y^*(z)>1+\varepsilon(1-\beta).
$$
Now, given $u,v\in M, u\neq v$ such that $\Vert \hat{\phi}(m_{u,v})\Vert>1+\varepsilon(1-\beta),$ it follows that
$$1+\varepsilon(1-\beta)<\Vert \hat{f}(m_{u,v})\Vert +\varepsilon \vert \hat{h}(m_{u,v})\vert \leq 1+\varepsilon \vert \hat{h}(m_{u,v})\vert,$$
from where we get that $\hat{h}(m_{u,v})>1-\beta$ or $\hat{h}(-m_{u,v})=\hat{h}(m_{v,u})>1-\beta$. Assume that $\hat{h}(m_{u,v})>1-\beta$ (the other case runs similarly).  This implies that $m_{u,v}\in S$, hence $\Vert m_{u,v}-m_{x,y}\Vert<\delta$. Now, by using \cite[Lemma 1.3]{ccgmr} we obtain that
$$\frac{\max\{d(x,u),d(y,v)\}}{d(x,y)}\leq \Vert m_{x,y}-m_{u,v}\Vert<\delta,$$
so $\max\{d(x,u),d(y,v)\}<\delta d(x,y)$. Hence
$$d(u,v)\geq d(x,y)-d(x,u)-d(y,v)>(1-2\delta)d(x,y),$$
from where we deduce that $\phi$ does not approximate its Lipschitz constant at arbitrarily close points, as desired.
\end{proof}

Next, we also need the following lemma, whose proof is encoded in \cite[Proposition 2.8.b]{ikw} for the real-valued case.

\begin{lemma}\label{lemma:nolocalstrexp}
Let $M$ be a compact metric space, let $Y$ be a Banach space and let $f\in S_{\Lip(M,Y)}$ be a non-local Lipschitz map. Then, there exists a strongly exposed point $m_{x,y}\in \mathcal F(M)$ such that $\Vert \hat{f}(m_{x,y})\Vert =1$.
\end{lemma}

\begin{proof}
Since $f$ is not local, then an easy compactness argument yields that we can find a pair of different points $x,y\in M$ such that not only $\Vert \hat{f}(m_{x,y})\Vert =1$, but also if $0<d(u,v)<d(x,y)$ then $\Vert \hat{f}(m_{u,v})\Vert <1$. We claim that the pair $(x,y)$ fails property (Z). Indeed, assume by contradiction that $(x,y)$ has property (Z). Pick $y^*\in S_{Y^*}$ such that $[y^*\circ \hat{f}](m_{x,y})=1$. Then, for every $n\in\mathbb N$, there exists a point $z_n\in M\setminus\{x,y\}$ satisfying that
    $$d(x,z_n)+d(y,z_n)\leq d(x,y)+\frac{1}{n}\min\{d(x,z_n),d(y,z_n)\}.$$
Up to taking a subsequence, we may assume that $d(z_n,x)\leq d(z_n,y)$ for every $n\in\mathbb N$. Also, up to taking a further subsequence, we may assume by compactness that $\{z_n\}\longrightarrow z\in M$.
Now, we have two possibilities:
\begin{itemize}
    \item If $x\neq z$ then it is clear that $d(x,z)+d(y,z)=d(x,y)$, which implies that $[y^*\circ \hat{f}](m_{x,z})=1$ and, in particular, $f$ strongly attains its norm at the pair $(x,z)$. However, notice that
    $$d(x,z)\leq \frac{1}{2}(d(x,z)+d(y,z))=\frac{1}{2}d(x,y),$$
    which contradicts the minimality condition on $d(x,y)$.
    \item If $x=z$, then
    \begin{align*}
        \Vert \hat{f}(m_{x,z_n})\Vert &\geq [y^*\circ \hat{f}](m_{x,z_n})\\ &=[y^*\circ \hat{f}](m_{x,y})\frac{d(x,y)}{d(x,z_n)}-[y^*\circ \hat{f}](m_{z_n,y})\frac{d(z_n,y)}{d(x,z_n)}\\
        & \geq \frac{d(x,y)-d(z_n,y)}{d(x,z_n)}\geq  1-\frac{1}{n},
    \end{align*}
which entails a contradiction with the assumption that $f$ is not local.
\end{itemize}
Consequently, we get that the pair $(x,y)$ fails property (Z), so $m_{x,y}$ is a strongly exposed point by \cite[Theorem 5.4]{gprdauga}.
\end{proof}

The last preliminary result we present on the way to proving Theorem \ref{theo:centraestruct} deals with norm attaining operators on general Banach spaces.

\begin{proposition}\label{prop:strongoperator}
Let $X$ and $Y$ be Banach spaces. The following assertions are equivalent:
\begin{enumerate}
    \item\label{prop:strongoperator1} The set $\{T\in \mathcal{L}(X,Y)\colon  T\text{ attains its norm at a strongly exposed point}\}$ is dense in $\mathcal{L}(X,Y)$.
    \item\label{prop:strongoperator2} The set $\{T\in \mathcal{L}(X,Y)\colon  T\text{ is absolutely strongly exposing operator}\}$ is dense in $\mathcal{L}(X,Y)$.
\end{enumerate}
\end{proposition}

\begin{proof}
 \eqref{prop:strongoperator2} $\Rightarrow$ \eqref{prop:strongoperator1}. Pick an absolutely strongly exposing operator $T$ for $x\in S_X$, and let us prove that $x$ is strongly exposed. Let $y^*\in S_{Y^*}$ such that $y^*(Tx)=\|T\|$ and consider $x^*\in S_{X^*}$ such that $\|T\|x^*=T^*(y^*)$. If $\{x_n\}$ is a sequence in $B_X$ such that $x^*(x_n)\longrightarrow 1=x^*(x)$, then $$\|T(x_n)\|\geq y^*(Tx_n)=\|T\|x^*(x_n)\longrightarrow \|T\|,$$ so there is a subsequence $\{x_{n_k}\}$ converging to $x$ (it cannot converge to $-x$), showing that $x$ is strongly exposed by $x^*$.

 \eqref{prop:strongoperator1} $\Rightarrow$ \eqref{prop:strongoperator2}. Pick an operator $T\in \mathcal{L}(X,Y)$ which attains its norm at a strongly exposed point $x$, and let us find an absolutely strongly exposing operator $S$ such that $\Vert T-S\Vert<\varepsilon$. For this, pick a strongly exposing functional $f_{x}$ for $x$. Define
    $$S:=T+\varepsilon f_x\otimes T(x).$$
Note that $\Vert S-T\Vert<\varepsilon$ is obvious. Let us prove that $S$ is absolutely strongly exposing. To this end, it is clear that $\Vert S\Vert\leq 1+\varepsilon$. Also, we get that
$$1+\varepsilon=(1+\varepsilon)\Vert T(x)\Vert =\Vert S(x)\Vert,$$
from where $\Vert S\Vert=1+\varepsilon$. Pick a sequence $\{x_n\}\in S_X$ such that $\Vert S(x_n)\Vert\longrightarrow 1+\varepsilon$. Since $S=T+\varepsilon f_x\otimes T(x)$ this implies that $\vert f_x(x_n)\vert\longrightarrow 1$ from where we can find a subsequence $\{x_{n_k}\}$ such that $f_x(x_{n_k})\longrightarrow 1$ or $f_x(x_{n_k})\longrightarrow -1$. Making use of the fact that $f_x$ strongly exposes $x$, we get that $\{x_{n_k}\}\longrightarrow x$ or $\{x_{n_k}\}\longrightarrow -x$. By definition, $S$ is an absolutely strongly exposing operator, so we are done.
\end{proof}

We are now able to present the pending proof.

\begin{proof}[Proof of Theorem \ref{theo:centraestruct}]
(i)$\Rightarrow$(iii) follows from Proposition \ref{prop:centraextremos} and Lemma \ref{lemma:approxnoloc}. (iii)$\Rightarrow$(ii) follows by Lemma \ref{lemma:nolocalstrexp} and Proposition \ref{prop:strongoperator}. Finally, (ii)$\Rightarrow$(i) follows from the fact that every absolutely strongly exposing operator attains its norm at a strongly exposed point, so at a molecule of $\mathcal F(M)$.
\end{proof}

As a consequence of the techniques involved in the proofs of Theorems \ref{theorem:closedconvexhullextremepoints} and \ref{theo:centraestruct}, we get the following result, which improves Theorem~\ref{theorem:closedconvexhullextremepoints} in the compact case, and also provides an affirmative answer to (Q3) in this case.

\begin{theorem}		\label{theorem:clcvstrepcompact}
		Let $M$ be a compact metric space. If $\SA(M,\R)$ is dense in $\Lip(M,\R)$, then
		$$
		B_{\F{M}}=\overline{\co}\bigl(\strexp{B_{\mathcal F(M)}}\bigr).
		$$
	\end{theorem}

\begin{proof} Let $\Gamma = \strexp{B_{\mathcal F(M)}}$. Assume that $f\in \Lip(M,\R)$ is such that $\hat{f}$ attains its norm at an element of $\Mol{M}\cap \ext{B_{\mathcal F(M)}}$.
By Lemmata \ref{lemma:approxnoloc} and \ref{lemma:nolocalstrexp}, $\hat{f}$ can be approximated by elements in $\mathcal{L}(\mathcal{F}(M),\mathbb R)$ attaining their norms on $\Gamma$. Therefore, $\sup_{m\in\Gamma} |\hat{f}(m)| = \norm{f}_L$. Now, Lemma \ref{lemma:gamma} gives that $B_{\mathcal F(M)}= \overline{\co}(\Gamma)$, as desired.
\end{proof}

We would like to observe that this result somehow improves, in the case of Lipschitz-free spaces on compact metric spaces, another result by Lindenstrauss \cite{linds63}.

\begin{remark}\label{remark:Lidenstrauss-remark2}
{\slshape Let $M$ be a compact metric space. If $\SA(M,Y)$ is dense in $\mathcal{L}(\F{M},Y)$ for \emph{some} Banach space $Y$, then}
$$
B_{\F{M}}=\overline{\co}\bigl(\strexp{B_{\mathcal F(M)}}\bigr).
$$
Indeed, this follows from Theorem \ref{theorem:clcvstrepcompact}, as the density of $\SA(M,Y)$ in $\Lip(M,Y)$ for some $Y$ implies the density of $\SA(M,\R)$ in $\Lip(M,\R)$ by \cite[Proposition~4.2]{cm}.

Compare this result with the following one by Lindenstrauss \cite[Theorem 2.ii]{linds63}: if $X$ is a Banach space which admits a LUR renorming (for instance, if $X$ is separable) such that $\NA(X,Y)$ is dense in $\mathcal{L}(X,Y)$ for \emph{all} Banach spaces $Y$, then $B_X=\overline{\co}(\strexp{B_{X}})$.
\end{remark}

\begin{remark}
In \cite[Problem 6.7]{Godefroy-survey-2015} it is proposed to study for which compact metric spaces $M$ and Banach spaces $Y$ one has that $\SA(M,Y)$ is dense in $\Lip(M,Y)$. Note that a necessary condition is that $B_{\mathcal F(M)}=\overline{\co}\bigl(\strexp{B_{\mathcal F(M)}}\bigr)$, according to the previous remark. Note, however, that this is not a sufficient condition, as the metric space $\mathbb T$ shows, see Theorem \ref{theo:toro}.
\end{remark}

Notice that techniques similar to those of Lemma \ref{lemma:approxnoloc} can be used in the locally compact case to get the following result.

\begin{proposition}\label{prop:denseopen-locallycompact}
Let $M$ be a locally compact metric space and let $Y$ be a Banach space. Then, the following assertions are equivalent:
\begin{enumerate}
\item[(i)] The set $\bigl\{f\in \Lip(M,Y)\colon \hat{f}\text{ attains its norm at a denting point}\bigr\}$ is dense in $\Lip(M,Y)$.
\item[(ii)] The set of absolutely strongly exposing operators from $\mathcal F(M)$ to $Y$ is dense in $\mathcal{L}(\F{M},Y)$.
\item[(iii)] $\SA(M,Y)$ contains the open dense set $B$ of the Lipschitz maps $f\colon M\longrightarrow Y$ with the following property: there are $\eta>0$, $x,y\in M$ with $x\neq y$ and $r>0$ such that
\begin{itemize}
\item $B(x,r)$ and $B(y,r)$ are compact and disjoint, and,
\item $\Vert \hat{f}(m_{u,v})\Vert\leq \Vert f\Vert_L-\eta$ if $(u,v)\notin (B(x,r)\times B(y,r))\cup (B(y,r)\times B(x,r))$.
\end{itemize}
\end{enumerate}
In particular, in such a case, $B_{\mathcal F(M)}=\overline{\co}\bigl(\strexp{B_{\mathcal F(M)}}\bigr)$.
\end{proposition}

In particular, for locally compact metric spaces whose Lipschitz-free space has the RNP, the proposition above gives the following corollary, which extends Corollary \ref{coro:RNP-SNA-open}, since the set of absolutely strongly exposing operators from $\mathcal F(M)$ to $Y$ is dense in $\mathcal{L}(\F{M},Y)$ by \cite[Theorem~5]{bou77}.

\begin{corollary}\label{coro:RNPabiden}
Let $M$ be a locally compact metric space for which $\mathcal F(M)$ has the RNP and let $Y$ be a Banach space. Then,  $\SA(M,Y)$ (and so $\NA(\F{M},Y)$) contains an open dense set.
\end{corollary}

Observe that this applies to the main examples in the literature of metric spaces $M$ for which it is known that $\mathcal F(M)$ has the RNP, as the class of uniformly discrete metric spaces or the class of boundedly compact H\"older metric spaces.

\begin{proof}[Proof of Proposition \ref{prop:denseopen-locallycompact}]
Assume that
$$
A:=\bigl\{f\in \Lip(M,Y)\colon \hat{f}\text{ attains its norm at a denting point}\bigr\}
$$
is dense in $\Lip(M,Y)$. Pick $f\in A$ with $\Vert f\Vert_L=1$ and find a denting point $m_{x,y}\in \mathcal F(M)$ and an element $y^*\in S_{Y^*}$ such that $[y^*\circ \hat{f}](m_{x,y})=1$. Fix $0<\delta<\frac12$ and find a slice $S=S(B_{\mathcal F(M)},\hat{h},\beta)$ containing $m_{x,y}$ and such that $\diam(S)<\delta$. Select $z\in S_Y$ such that $y^*(z)\hat{h}(m_{x,y})>1-\beta$ and define
    $$\hat{\phi}:=\hat{f}+\varepsilon \hat{h}\otimes z.$$
    It is clear that $\Vert f-\phi\Vert_L\leq\varepsilon$. Also, the proof of Lemma \ref{lemma:approxnoloc} reveals that, given $u,v\in M, u\neq v$ then
    \begin{multline*}
    \Vert \hat{\phi}(m_{u,v})\Vert>1+\varepsilon(1-\beta)\\ \Longrightarrow (u,v)\in (B(x,\delta d(x,y))\times B(y,\delta d(x,y)))\cup (B(y,\delta d(x,y))\times B(x,\delta d(x,y))).
    \end{multline*}
Taking into account that $B(x,\delta d(x,y))$ and $B(y,\delta d(x,y))$ are compact and disjoint for a small enough $\delta$, we derive that $\phi\in B$. This proves that the set $B$ is dense. To get (iii) let us prove that $B$ enjoys the following properties:
\begin{enumerate}
\item {\slshape $B$ is open.}\ Indeed, given a map $f\in B$, consider $\eta>0$, $x,y\in M$ with $x\neq y$ and $r>0$ for which $B(x,r)$ and $B(y,r)$ are compact and disjoint, and $\Vert \hat{f}(m_{u,v})\Vert\leq \norm{f}-\eta$ if $(u,v)\notin (B(x,r)\times B(y,r))\cup (B(y,r)\times B(x,r))$. Pick $0<\delta<\frac{\eta}{2}$ and let us prove that $B(f,\delta)\subseteq B$. To this end take $g\in \Lip(M,Y)$ with $\Vert f-g\Vert_L<\delta$. Now, if $(u,v)\notin (B(x,r)\times B(y,r))\cup (B(y,r)\times B(x,r))$ then $\Vert \hat{f}(m_{u,v})\Vert\leq \Vert f\Vert_L-\eta$, from where
$$\Vert \hat{g}(m_{u,v})\Vert\leq \delta+\Vert \hat{f}(m_{u,v})\Vert\leq \delta+\Vert f\Vert_L-\eta\leq \Vert g\Vert_L+2\delta-\eta,$$
which proves that $g\in B$, as desired.
\item {\slshape Every map in $B$ attains its norm at a strongly exposed point.} To see this, take $f\in B$ and, by definition, consider $\eta>0$, $x,y\in M$ with $x\neq y$ and $r>0$ for which $B(x,r)$ and $B(y,r)$ are compact and disjoint, and $\Vert \hat{f}(m_{u,v})\Vert\leq \Vert f\Vert_L-\eta$ if $(u,v)\notin (B(x,r)\times B(y,r))\cup (B(y,r)\times B(x,r))$. Notice that $f$ strongly attains its norm because the set $B(x,r)\cup B(y,r)$ is a compact set and from the fact that $f$ cannot approximate its norm at arbitrarily close points. The previous fact even provides a pair of different points $u\in B(x,r)$ and $v\in B(y,r)$ with the property that $\Vert \hat{f}(m_{u,v})\Vert=\Vert f\Vert_L$. The pair $(u,v)$ fails property (Z). Indeed, otherwise there would be a sequence $\{z_n\}$ such that $\frac{(u,v)_{z_n}}{\min\{d(u,z_n), d(v,z_n)\}}\to 0$ and so, by \cite[Lemma~3.7]{ccgmr}, it would follow that $z_n\in B(x,r)\cap B(y,r)=\emptyset$ for large $n$, a contradiction. Equivalently, the molecule $m_{u,v}$ is strongly exposed point. This proves that (iii) implies (ii) by Proposition \ref{prop:strongoperator}.
\end{enumerate}
Now the previous two facts prove that (i) implies (iii). Finally, (ii) implies (i) is trivial, which finishes the proof.
\end{proof}

Apart from Corollary \ref{coro:RNPabiden}, Proposition \ref{prop:denseopen-locallycompact} also applies to another large class of metric spaces.

\begin{example}
{\slshape Let $M$ be a locally compact metric space, let $0<\theta<1$, and consider $M^\theta:=(M,d^\theta)$. Then, $\SA(M^\theta,Y)$ contains an open dense subset.}\ Indeed, $M^\theta$ is locally compact and $\SA(M^\theta,Y)$ is dense in $\Lip(M^\theta,Y)$ for every Banach space $Y$ by \cite[Proposition~3.8]{cm}; now, Proposition \ref{prop:denseopen-locallycompact} applies since every molecule of $\mathcal F(M^\theta)$ is a strongly exposed point \cite[Proposition~3.8]{cm}.
\end{example}

Let us end by giving a generalisation of Theorem~\ref{theorem:clcvstrepcompact}.

\begin{corollary}\label{coro:clcvstrepbouncompact}
Let $M$ be a boundedly compact metric space. If $\SA(M,\R)$ is dense in $\Lip(M,\R)$, then
		$$
B_{\F{M}}=\overline{\co}\bigl(\strexp{B_{\mathcal F(M)}}\bigr).
		$$
\end{corollary}

\begin{proof} As in Theorem \ref{theorem:clcvstrepcompact}, let $\Gamma = \strexp{B_{\mathcal F(M)}}$ and suppose $f\in \Lip(M,\R)$ is such that $\hat{f}$ attains its norm at an element $m_{p,q} \in\ext{B_{\mathcal F(M)}}$. Since $M$ is a boundedly compact metric space, by using the techniques involved in the proof of Theorem 4.2 in \cite{ag} and Theorem 2.4 in \cite{lppr}, we obtain that $m_{p,q}$ is a denting point of $B_{\mathcal F(M)}$. Now, it follows from the proof of Proposition \ref{prop:denseopen-locallycompact} that $\hat{f}$ can be approximated by elements of $\mathcal{L}(\mathcal{F}(M),\R)$ attaining their norms on $\Gamma$. Therefore, $\sup_{m\in\Gamma} |\hat{f}(m)|=\norm{f}_L$. Now, Lemma \ref{lemma:gamma} does the work.
\end{proof}

\section{A by-product}\label{sect:byproduct}
We finish the paper with a by-product of the preliminary results used in the proof of Theorem \ref{theo:centraestruct}. We do not know whether Lindenstrauss property A of a Banach space $X$ implies that the set
\begin{equation}\tag{\ensuremath{\clubsuit}}\label{eq:abs-stron-exposing-Remark}
\{T\in \mathcal{L}(X,Y)\colon  T\text{ is absolutely strongly exposing operator}\}
\end{equation}
is dense in $\mathcal{L}(X,Y)$ for every Banach space $Y$. But we would like to check what happens if we actually have one of the known properties which imply Lindenstrauss property A: the RNP, property $\alpha$, property quasi-$\alpha$, and having a norming uniformly strongly exposed set. It is known that the answer is positive if $X$ has the RNP by \cite[Theorem 5]{bou77}. Also, the same follows if $X$ has property $\alpha$ by an easy inspection of the proof of \cite[Proposition 1.3a]{schachermayer}. But this is not clear from the proof of the result for property quasi-$\alpha$ and when the space has a norming uniformly strongly exposed set. However, in the first case it is shown in \cite[Proposition 2.1]{ChoiSong} that if $X$ has property quasi-$\alpha$ then, for every Banach space $Y$, the set of operators attaining the norm at a strongly exposed point is dense in $\mathcal{L}(X,Y)$ (see also the comments after Definition~3.18 in \cite{ccgmr}). Therefore, by Proposition \ref{prop:strongoperator} we get that the set in \eqref{eq:abs-stron-exposing-Remark} is dense. In the case when $X$ is a Banach space with a uniformly strongly exposed set $S\subset S_X$ such that $B_X=\overline{\co}(S)$, it is shown in (the proof of) \cite[Proposition~1]{linds63} that the set
$$
\left\{T\in \mathcal{L}(X,Y)\colon T\text{ attains its norm at a point of }\overline{S}\right\}
$$
is dense in $\mathcal{L}(X,Y)$ for every Banach space $Y$. Then, the density of the set in \eqref{eq:abs-stron-exposing-Remark} follows again from Proposition \ref{prop:strongoperator} and the fact that $\overline{S}$ is a uniformly strongly exposed set too:

\begin{fact}\label{fact:cufeabsostronglyope}
Let $X$ be a Banach space and assume that $S\subset S_X$ is a uniformly strongly exposed set. Then, $\overline{S}$ is also a uniformly strongly exposed set.
\end{fact}

\begin{proof}
For every $y\in S$, let $f_y$ be the corresponding strongly exposing functional. Clearly, given $x\in \overline{S}\setminus S$, we can find a functional $f_x$ in the weak-star closure of $\{f_y \colon y\in S\}$ such that $f_x(x)=1$. Let us now prove that $\overline{S}$ is uniformly strongly exposed. To this end, for $\varepsilon>0$ we find $\delta>0$ such that
    $$
    f_y(x)>1-\delta\ \ \Longrightarrow\ \ \Vert y-x\Vert<\frac{\varepsilon}{2}
    $$
for every $y\in S$. Now, given $x\in \overline{S}$ and $z\in S_X$ such that $f_x(z)>1-\delta$, we find $y\in S$ such that $f_y(z)>1-\delta$ and $f_y(x)>1-\delta$ (recall that $f_x$ belongs to the weak-star closure of $\{f_y\colon  y\in S\}$). The property defining $\delta$ implies that $\Vert y-z\Vert<\frac{\varepsilon}{2}$ and $\Vert x-y\Vert<\frac{\varepsilon}{2}$, from where $\Vert x-y\Vert<\varepsilon$.
\end{proof}

Let us state what has been proved so far.

\begin{proposition}\label{prop-byproduct}
Let $X$ be a Banach space satisfying one of the following properties:
\begin{enumerate}
\item[(a)] RNP,
\item[(b)] property $\alpha$,
\item[(c)] property quasi-$\alpha$,
\item[(d)] having a norming uniformly strongly exposed subset.
\end{enumerate}
Then, the set
$$
\{T\in \mathcal{L}(X,Y)\colon  T\text{ is absolutely strongly exposing operator}\}
$$
is dense in $\mathcal{L}(X,Y)$ for every Banach space $Y$.
\end{proposition}

\section*{Acknowledgment} The authors are deeply grateful to F.~Nazarov for suggesting the set $C$ used in Lemma~\ref{lema:nazarov}. They also thank Antonio Avil\'es, Gilles Godefroy, and Rafael Pay\'{a} for kindly answering several inquiries related to the topic of the paper, and thank Ram\'on Aliaga for some interesting remarks on the first version of the manuscript which resulted in an improved exposition. They also thank the anonymous referee for the careful reading of the paper and valuable suggestions.

\end{document}